\documentclass[a4paper,11pt,french,francais,twoside,fleqno]{amsart}

\usepackage[T1]{fontenc}
\usepackage[latin1]{inputenc}
\usepackage{pifont}
\usepackage{euscript}
\usepackage{mathrsfs}
\usepackage{lmodern}
\rmfamily
\DeclareFontShape{T1}{lmr}{b}{sc}{<->ssub*cmr/bx/sc}{}
\DeclareFontShape{T1}{lmr}{bx}{sc}{<->ssub*cmr/bx/sc}{}

\DeclareMathAlphabet{\mathpzc}{OT1}{pzc}{m}{it}

\usepackage[textwidth=14cm,hcentering]{geometry}
\DeclareRobustCommand{\SkipTocEntry}[4]{}
\usepackage{multicol}
\usepackage{enumerate}
\usepackage{calc}

\usepackage[all]{xy}
\usepackage{amsthm,amssymb,amsmath}
\usepackage[fixamsmath]{mathtools}

\usepackage[frenchb]{babel}
\usepackage{xspace}
\usepackage{enumitem} 

\newcommand{\cal}{\mathcal}

\newcommand{\GroT}{\widehat{GT}}

\newcommand{\Gq}{\Gal(\overline{\QQ}/\QQ)}
\newcommand{\mcg}{\Gamma_{g,[m]}}

\newcommand{\NN}{{\mathbb N}}
\newcommand{\ZZ}{{\mathbb Z}}
\newcommand{\QQ}{{\mathbb Q}}

\newcommand{\CC}{{\mathbb C}}

\renewcommand{\O}{{\cal O}}

\newcommand{\Y}{{\cal Y}}
\newcommand{\M}{{\cal M}}

\newcommand{\GG}{\ensuremath{\mathbb G}}
\newcommand{\mmu}{\ensuremath{\boldsymbol \mu}}

\newcommand{\Spec}{{\operatorname{Spec}\kern 1pt}}
\newcommand{\Spf}{{\operatorname{Spf}\kern 1pt}}

\newcommand{\Pic}{{\mathrm{Pic}}}

\renewcommand{\mod}{\mathrm{\,mod\,}}

\def\buildrel#1\over#2{\mathrel{\mathop{\kern 0pt#2}\limits^#1}}

\renewcommand{\H}{{\mathrm{H}}}
\newcommand{\R}{{\mathrm{R}}}
\renewcommand{\Im}{\operatorname{Im}\kern 1pt}
\newcommand{\coker}{\operatorname{coker}\kern 1pt}

\newcommand{\Gal}{{\mathrm{Gal}}}
\newcommand{\card}{{\mathrm{card}}}

\renewcommand{\epsilon}{\varepsilon}
\newcommand{\Aut}{{\mathrm{Aut}}}

\newcommand{\Ar}{{\mathbf{Ar}\kern 0.5pt}}

\newcommand{\supp}{{\mathrm{supp}}}
\renewcommand{\div}{{\mathrm{div}}}

\newcommand{\cd}{\operatorname{cd}\kern 1pt}

\renewcommand{\tilde}{\widetilde}

\newcommand{\X}{{\mathcal X}}

\newcommand{\U}{{\mathcal U}}

\newcommand{\et}{\text{\rm ét}}

\newcommand{\cart}{\ar@{}[dr] |{\square}}
\newcommand{\comm}{\ar@{}[dr] |{\circlearrowleft}}
\newcommand{\incl}[1][r]{\ar@<-0.2pc>@{^(-}[#1] \ar@<+0.2pc>@{-}[#1]}

\newcommand{\isomto}{\overset{\sim}{\rightarrow}}

\numberwithin{equation}{section}

\theoremstyle{plain}
\newtheorem{theo}{Th\'eor\`eme}[section]
\newtheorem{defi}[theo]{D\'efinition}
\newtheorem{cor}[theo]{Corollaire}
\newtheorem{prop}[theo]{Proposition}

\newtheorem{lem}[theo]{Lemme}

\theoremstyle{remark}
\newtheorem{rem}[theo]{Remarque}
\newtheorem{exem}[theo]{Exemple}
\makeatletter
\let\@altabstract\@empty
\let\@alttitle\@empty
\newcommand*{\alttitle}[1]{\gdef\@alttitle{#1}}%
\newcommand*{\altkeywords}[1]{\gdef\@altkeywords{#1}}%
\newbox\altabstractbox
\newenvironment{altabstract}{%
	\ifx\maketitle\relax
	\ClassWarning{\@classname}{Altabstract should precede
		\protect\maketitle\space in AMS document classes; reported}%
	\fi
	\global\setbox\altabstractbox=\vtop \bgroup
	\normalfont\Small
	\list{}{\labelwidth\z@
		\leftmargin3pc \rightmargin\leftmargin
		\listparindent\normalparindent \itemindent\z@
		\parsep\z@ \@plus\p@
		
	}%
	\item[\hskip\labelsep\scshape\altabstractname.]%
}{%
\endlist\egroup
\ifx\@setaltabstract\relax \@setaltabstracta \fi
}
\def\@setaltabstract{\@setaltabstracta \global\let\@setaltabstract\relax}
\def\@setaltabstracta{%
	\ifx\@empty\@alttitle\else
	\begin{center}%
		\large\slshape  \@alttitle  
	\end{center}%
	\fi
	\ifvoid\altabstractbox
	\else
	\box\altabstractbox
	\prevdepth\z@ 
	\fi
}
\def\@setabstract{\@setabstracta \global\let\@setabstract\relax
	\@setaltabstract}

\newcommand{\altabstractname}{Abstract}
\makeatother

\title[Lieux spéciaux de $\M_{g,[m]}$, actions de Galois absolu]{Composantes irréductibles de lieux spéciaux d'espaces de modules de courbes, action galoisienne en genre quelconque}

\author[B.~Collas et S.~Maugeais]{Benjamin Collas et Sylvain Maugeais}

\address{Institut de Mathématiques de Jussieu\\
	Université Pierre et Marie Curie - Paris 6\\
	4, place Jussieu \\
	75 254 Paris Cedex 5 (France)}
\email{collas@math.cnrs.fr}

\thanks{B. Collas remercie le \emph{laboratoire Manceau de Mathématiques} de l'université du Maine pour son accueil qui a permis cette collaboration, ainsi que M. S. Weiss pour son financement via sa bourse professorale Humboldt. Les auteurs remercient P.~Lochak et M.~Vaquié pour les nombreuses discussions à l'origine de ce travail, P.~Dèbes pour ses remarques, ainsi que le rapporteur dont les commentaires ont améliorés la précision et la clareté du texte.}

\address{Universit\'e du Maine\\
	Laboratoire manceau de Mathématiques\\
	Av. Olivier Messiaen, BP 535 \\
	72017 Le Mans Cedex (France)}
\email{sylvain.maugeais@univ-lemans.fr}

\altkeywords{algebraic fundamental group, stack inertia, special loci, good groups}



\usepackage[a4paper,colorlinks=true, pdfstartview=FitV, linkcolor=blue, citecolor=blue, urlcolor=blue]{hyperref}

\hypersetup{  
	pdfinfo={  
		Title={Lieux spéciaux d'espaces de modules de courbes marquées, actions de Galois absolu},  
		Author={Benjamin Collas, Sylvain Maugeais},
		Subject={Dans cet article, nous caractérisons l'action du groupe de Galois absolu sur les groupes d'inertie champêtre géométriques cycliques et sans factorisation étale du groupe fondamental géométrique des espaces de modules de courbes marquées. Nous établissons par ailleurs la même action sur les éléments de torsion profinis d'ordre premier en genre 2},  
		Keywords={groupe fondamental algébrique, inertie champêtre, lieu spécial, groupes bons},  
	}  
}

\begin{document}
	
	\begin{abstract}
		Dans cet article, nous caractérisons l'action du groupe de Galois absolu sur les groupes d'inertie champêtre géométriques cycliques et sans factorisation étale du groupe fondamental géométrique des espaces de modules de courbes marquées. Nous établissons par ailleurs la même action sur les éléments de torsion profinis d'ordre premier en genre $2$.
	\end{abstract}
	
	\begin{altabstract}
		In this paper we characterise the action of the absolute Galois group on the geometric finite cyclic groups without étale factorization of stack inertia of the profinite geometric fundamental group of moduli spaces of marked curves. As a complementary result, we give the same action on prime order profinite elements in genus $2$.
	\end{altabstract}
	
	\maketitle
	
	\section{Introduction}
	Nous suivons dans cet article l'approche générale d'A.~Grothendieck d'\emph{Esquisse d'un programme} \cite{GRO97} qui consiste à étudier le groupe de Galois absolu $\Gq$ à travers sa représentation géométrique dans des groupes fondamentaux d'espaces de modules de courbes marquées et non-ordonnées $\M_{g,[m]}$ :
	\begin{equation*}
		\Gq\rightarrow Out(\pi_1^{geom}(\mathcal{M}_{g,[m]})).
	\end{equation*}
	Les espaces de modules $\mathcal{M}_{g,[m]}$ étant munis de leur structure de $\mathbb Q$-champ algébrique de Deligne-Mumford \cite{DEL69}, cette représentation est établie par T.~Oda \cite{ODA}. Le groupe fondamental de $\mathcal{M}_{g,[m]}$ contient par ailleurs certains sous-groupes définis par les structures \emph{inertielle à l'infini} et \emph{inertielle champêtre} de cet espace.
	
	\bigskip
	
	Nos résultats concernent l'action de $\Gq$ sur l'inertie champêtre géométrique, i.e. \emph{finie}, et sans factorisation étale du groupe profini $\pi_1^{geom}(\M_{g,[m]})$ pour tout genre $g\geqslant 0$ et pour un nombre $m\geqslant 0$ quelconque de points marqués. Nous donnons par ailleurs l'action galoisienne sur un certain type de torsion \emph{profinie d'ordre premier} en genre $g=2$. 
	
	La difficulté principale consiste à établir l'invariance des classes de conjugaison d'inertie sous l'action galoisienne. Dans le cas de l'inertie champêtre, nous résolvons ce problème par la caractérisation des composantes irréductibles du champ des espaces de modules des revêtements cycliques.
	
	Dans le cas du genre $2$, nous réduisons la torsion profinie à la torsion finie et le résultat découle du résultat précédent.

	\subsection{Inerties des espaces de modules de courbes, action galoisienne}
	Considérons la compactification de Deligne-Mumford $\overline{\mathcal{M}}_{g,m}$ de l'espace \emph{ordonné} $\mathcal{M}_{g,m}$ ainsi que son diviseur à l'infini $D_{\infty}=\overline\M_{g,m}\setminus \M_{g,m}$. À une composante irréductible $D$ de $D_{\infty}$
	est associé un groupe d'inertie $I_D\subset \pi_1^{geom}(\M_{g,m})$, cyclique suivant la théorie du groupe fondamental modéré \cite{GRO71}.
	
	Dans un contexte de théorie généralisée de Grothendieck-Murre H.~Nakamura a alors établi que \emph{l'action du groupe de Galois absolu $\Gq$ sur \emph{certains} groupes d'inertie à l'infini des espaces de modules $\M_{g,m}$ est donnée par $\chi(\sigma)$-conjugaison}, c'est-à-dire conjugaison d'un générateur d'inertie et élévation à la puissance cyclotomique $\chi(\sigma)$ -- voir \cite{NAKA94} pour le cas du genre zéro et \cite{NAKA97,NAKA99} pour le cas général. Ce résultat a été étendu à l'intégralité de ces groupes d'inertie grâce à la théorie de Grothendieck-Teichmüller dans \cite{Nakamura2000}. 
	
	Dans le cas des espaces de modules $\M_{0,[m]}$ non-ordonnés de genre zéro, la théorie de Grothendieck-Teichmüller permet de plus d'établir un résultat similaire à partir d'expressions en terme de tresses d'Artin des générateurs d'inertie -- voir introduction de \cite{Colg0}. 
	
	\bigskip
	
	D'autre part, à la structure de champ des espaces $\mathcal{M}_{g,[m]}$ sont associés des  \emph{groupes d'inertie champêtre géométriques} $I_x=\Aut(x)$ composés des automorphismes finis des classes d'isomorphisme d'objet $x\in\mathcal{M}_{g,[m]}$. Suivant B.~Noohi \cite{NOOHI04} ces groupes d'automorphismes \emph{géométriques} s'injectent dans le groupe fondamental 
	\begin{equation*}
		\omega_x:I_x\to\pi_1^{geom}(\mathcal{M}_{g,[m]},x).
	\end{equation*}
	Rappelons que $\pi_1^{geom}(\mathcal{M}_{g,[m]})$ est isomorphe au complété profini du groupe fondamental orbifold $\pi_1^{orb}((\M_{g,[m]})_\mathbb C^{an})$. Identifiant ce dernier avec le mapping class group $\mcg$, groupe des difféormorphimes de surfaces, le théorème de réalisation de Nielsen-Kerckhoff \cite{KERC} donne une correspondance bijective entre les sous-groupes finis géométriques de $\pi_1^{geom}(\M_{g,[m]})$ et les sous-groupes finis du mapping class group $\Gamma_{g,[m]}$. Nous utiliserons cette correspondance dans la suite.
	
	\bigskip
	
	Une approche par théorie de Grothendieck-Teichmüller sur la torsion \emph{profinie} d'ordre premier a ainsi montré que \emph{l'action de $\Gq$ sur l'inertie champêtre géométrique d'ordre premier des espaces de modules $\M_{0,[m]}$ et $\M_{1,[m]}$ est donnée par $\chi(\sigma)$-conjugaison} -- voir \cite{ Colg0, Colg1}. Il s'agit ici de compléter cette comparaison d'action galoisienne entre ces groupes d'inerties à l'infini et champêtre.

	L'objectif principal de cet article est de généraliser ce résultat aux espaces de modules de courbes de genre quelconque ainsi que de s'affranchir de l'hypothèse de primalité, ce que nous obtenons sous une certaine condition géométrique d'absence de factorisation étale (voir Théorème 5.4 pour une définition) :
	
	\medskip
	
	{\noindent\textsc{Théorème A} ---
		\itshape Soit $G$ un groupe d'inertie champêtre géométrique cyclique de $\pi_1^{geom}(\M_{g,[m]})$ n'admettant pas de factorisation étale, pour un genre $g$ et un nombre de points marqués $m$ quelconques. L'action du groupe de Galois absolu $\Gq$ sur $G$ est donnée par $\chi(\sigma)$-conjugaison.
	}
	
	\medskip
	
	Du point de vue galoisien, ce résultat complète ainsi en partie la condition nécessaire liée à la question suivante  -- voir question 8.5 de \cite{LOC11} : \emph{les multitwists profinis sont-ils entièrement caractérisés par une action galoisienne donnée par $\chi(\sigma)$-conjugaison, à inertie près ?}
	
	\bigskip
	
	Pour un groupe $G$ fini cyclique fixé, notre approche consiste à caractériser les composantes irréductibles du \emph{lieu spécial} $\M_{g,[m]}(G)$, sous-champ de $\M_{g,[m]}$ dont les points géométriques admettent $G$ comme groupe d'inertie géométrique. Cette question se réduisant au cas des champs $\M_{g,[m]}[G]$ des courbes marquées avec $G$-autmorphismes -- voir section \ref{sec:lieuxSpec} --, nous donnons une construction algébrique des invariants de Hurwitz de M.~Cornalba \cite{COR87,CORER08} que nous affinons dans le cas des $G$-revêtements marqués en  donnée de branchement -- voir section \ref{sec:brData}. Nous construisons ainsi des $G$-revêtements de donnée de branchement fixée, dont nous déduisons une décomposition en sous-champs géométriquement irréductibles en section \ref{sec:CompIrr}
	\[
	\M_{g,[m]}[G]/\Aut(G)=\coprod\M_{g,[m],\underline{\bf kr}}[G]/\Aut(G). 
	\]
	
	La caractérisation de ces composantes nous permet alors d'établir leur invariance sous action galoisienne puis la formule de $\chi(\sigma)$-conjugaison souhaitée -- voir section \ref{sec:actionGal}.

	\subsection{Torsion profinie et action galoisienne}\label{sec:IntroGT}
	À la différence de l'inertie géométrique, la torsion profinie n'est \emph{a priori} pas reliée à des composantes irréductibles. Ainsi, suivant \cite{LS06}, questions 3.6 et 3.5, se pose la question d'exprimer l'action galoisienne sur la torsion \emph{profinie} du groupe fondamental $\pi_1^{geom}(\M_{g,[m]})$ à partir des groupes d'inertie \emph{champêtre} de celui-ci.
	
	\bigskip
	
	Dans un contexte de théorie cohomologique de Serre introduite dans \cite{SERREL}, \cite{Colg0, Colg1} montrent ainsi \emph{en genres zéro et un} que la torsion \emph{profinie d'ordre premier} de $\pi_1^{geom}(\M_{g,[m]})$ est conjuguée à la torsion géométrique. Le résultat déjà cité d'action galoisienne sur ces éléments s'obtient ensuite dans un contexte de théorie de Grothendieck-Teichmüller -- voir \emph{ibid}.

	Plus précisément, le groupe de Grothendieck-Teichmüller $\GroT$ introduit par V.~G.~Drinfel'd \cite{DRI90} agit sur $\pi_1^{geom}(\M_{0,[m]})$ et contient $\Gq$. Par ailleurs, Y.~Ihara a montré \cite{IHARA200} qu'il existe un point base tangentiel $\overrightarrow{b}$ en lequel baser le groupe fondamental de sorte que l'action de $\Gq$ s'étendent à $\GroT$:
	\[\xymatrix{
		\Gq\ar@{^(->}[rr]\ar@{^(->}[dr]& &\Aut(\pi_1^{geom}(\M_{0,[m]},\overrightarrow{b}))\\
		&\GroT\ar@{^(->}[ur]&
	}\]
	
	De cette compatibilité découle que le résultat d'action de $\GroT$ par $\lambda$-conjugaison sur la torsion cyclique profinie d'ordre premier établi dans \cite{Colg0} se transmet à $\Gq$ sur ces mêmes éléments. Nous étendons ici partiellement ces résultats au cas des espaces de modules de courbes de genre deux.
	
	\medskip
	
	{\noindent\textsc{Théorème B} ---  
		\itshape Le groupe de Galois absolu $\Gq$ agit sur la torsion profinie d'ordre premier sans factorisation étale de $\pi_1^{geom}(\M_{2,[m]})$ par $\chi(\sigma)$-conjugaison.
	}
	
	\medskip

	Ce résultat repose sur la bonté du mapping class group $\Gamma_{2,[m]}$ ainsi que sur un résultat cohomologique de P.~Symonds qui réduit les classes de conjugaisons de $p$-torsion profinie aux classes de conjugaisons de $p$-torsion finis, i.e. géométriques. Le résultat d'action est alors une conséquence directe du résultat d'action de $\Gq$ sur l'inertie géométrique établie au théorème A.

	\section{Lieux spéciaux, mapping class group, action galoisienne}\label{sec:lieuxSpec}
	Considérons les espaces des modules de courbes \emph{marquées} $\M_{g,[m]}$. Nous donnons ici la définition algébrique \emph{en famille} du champ des lieux spéciaux $\M_{g,[m]}(G)$ en reprenant les définitions de \cite{FRED03} établies pour les espaces $\M_{g}$. Rapprochant ces espaces du champ des espaces de modules des courbes avec $G$-automorphismes $\M_g[G]$ étudiés dans \cite{TUF93,EKE95,MAU06}, et des classes de conjugaison de $G$ dans le mapping class group $\mcg$, nous réduisons l'étude de l'action galoisienne sur les composantes irréductibles du champ des lieux spéciaux $\M_{g,[m]}(G)$ à l'étude de cette action sur les composantes irréductibles de $\M_{g,[m]}[G]/\Aut(G)$.

	\subsection{Lieux spéciaux de $\M_{g,[m]}$, espace de modules de courbes avec $G$-automorphismes}
	Considérons le cas de courbes hyperboliques de genre $g$ à $m$ points marqués, \emph{i.e.} telles que $2g-2+m>0$, et notons $\M_{g,[m]}$ \emph{l'espace des modules de courbes à points marqués} qui classifie les familles de courbes de genre $g$ à $m$ points marqués \emph{à permutation près} : ses objets sont donnés par des familles de courbes $\X/S$ de genre $g$ sur un $\QQ$-schéma $S$, i.e. un morphisme
	\[
	p:\X\to S\text{ propre lisse, où } \X \text{ est une courbe de genre } g
	\]
	dont les fibres géométriques $\X_s$ sont des courbes connexes de genres $g$, où $\X/S$ est muni d'un $m$-\emph{marquage}; i.e. un diviseur de Cartier effectif $D$ étale fini de degré $m$ sur $S$.

	Remarquons que cette notion de marquage étend celle de pointage qui définit l'espace des modules de courbes \emph{ordonnées} $\M_{g,m}$ où les points sont donnés par des sections $(\sigma_i:S\to \X)_{i=1\dots m}$. La donnée d'un marquage correspond à accepter la permutation des points, puisque
	\[\M_{g,[m]}=\M_{g,m}/\mathfrak{S}_m.\]
	où le quotient est pris dans la catégorie des champs.
	
	Suivant \cite{KNUII}, les courbes étant hyperboliques, l'espace de modules $\M_{g,m}$ admet une structure de $\ZZ$-champ algébrique de Deligne-Mumford, et il en est de même pour $\M_{g,[m]}$ comme quotient fini de celui-ci.

	\bigskip
	
	Par ailleurs, une famille de courbes $\X/S$ admet une action fidèle d'un groupe $G$ lorsqu'il existe un morphisme de schémas en groupes 
	\[
	G\to \Aut_S(\X) \text{ injectif sur les fibres.}
	\]
	Le champ des courbes avec $G$-automorphismes suivant, étudié notamment par S.~Tufféry et T.~Ekedahl \cite{TUF93,EKE95} dans le cas non marqué, est un outil essentiel pour la description de $\M_{g,[m]}(G)$ et de ses composantes irréductibles.
	\begin{defi}
		Soit $G$ un groupe fini. Le champ $\M_{g,[m]}[G]$ des courbes avec $G$-automorphismes est la catégorie fibrée en groupoïde sur $Sch/\QQ$ classifiant les familles de courbes $m$-marquées de genre $g$ munies d'une action fidèle de $G$ permutant les points. 
	\end{defi}
	Les objets de $\M_{g,[m]}[G]$ sont des triplets $(\X/S,\iota\colon G\rightarrow \Aut_S(\X),D)$ où $D$ est un diviseur étale de Cartier relatif fixe par $G$ de degré $m$; les morphismes sont $G$-équivariants et préservent $D$.

	Dans la suite, nous nous intéressons aux composantes irréductibles du champ suivant.
	\begin{defi}
		On nomme \emph{champ des lieux spéciaux} et l'on note $\M_{g,[m]}(G)$ la sous-catégorie de $\M_{g,[m]}$ fibrée en groupoïde constituée des familles de courbes $\X/S$ dont les fibres \emph{géométriques} admettent une action fidèle de $G$ permutant les points.
	\end{defi}

	Considérons le mapping class group $\mcg:=\pi_0(\text{Diff}^+(S_{g,[m]}))$, composante connexe du groupe des difféormorphimes de la surface topologique $S_{g,m}$ qui préservent l'orientation, avec permutation permise des points. Pour un sous-groupe $G$ fini de $\mcg$ fixé, le champ des lieux spéciaux est non vide puisque, selon le théorème de réalisation de Nielsen-Kerckhoff \cite{KERC}, pour tout sous-groupe fini $G$ de $\Gamma_{g,[m]}$, il existe une structure conforme sur la surface topologique $X$ de type $(g,m)$ tel que $G$ est le groupe des automorphismes conformes $\Aut(X)$.

	Notons que $\M_{g,[m]}[G]$ n'est pas un sous-champ de $\M_{g,[m]}$. Plus précisément, le champ des lieux spéciaux $\M_{g,[m]}(G)\subset \M_{g,[m]}$ est l'image de $\M_{g,[m]}[G]$ par le morphisme d'oubli de l'action de $G$.
	\begin{prop}
		Le champ des lieux spéciaux $\M_{g,[m]}(G)$ est un $\QQ$-champ de Deligne-Mumford. 
	\end{prop}
	
	\begin{proof}
		C'est immédiat puisque le morphisme d'oubli précédent $\M_{g,[m]}[G]\to\M_{g,[m]}$ est représentable et propre, donc son image $\M_{g,[m]}(G)$ est fermée, et $\M_{g,[m]}(G)$ hérite de la structure de champ de $\M_{g,[m]}[G]$.
	\end{proof}
	
	Le champ des lieux spéciaux $\M_{g,[m]}(G)$ correspond aux \emph{lieux $G$-symétriques} introduits par S.~Broughton \cite{BRO90} dans le contexte analytique des espaces de Teich\-müller et définis comme l'image dans $\left(\M_{g}\right)_{\mathbb C}$ des points de $\mathcal{T}_{g}$ stables par un conjugué de $G$. En particulier, le résultat suivant découle directement d'un résultat de théorie de Teichmüller \cite{Gonzalez92} -- voir \cite{Rom09} pour une preuve purement algébrique.

	\begin{prop}\label{prop:norm}
		Le champ $\M_{g,[m]}[G]/\Aut(G)$ est le normalisé du champ des lieux spéciaux $\M_{g,[m]}(G)$.  
	\end{prop}

	Le résultat suivant en découle alors immédiatement.
	\begin{cor}\label{prop:irrSpecIrrHur}
		Soit $\M_{g,[m]}(G)$ le champ des lieux spéciaux associé à un groupe $G$ fini. Alors les composantes irréductibles de $\M_{g,[m]}(G)$ sont en bijection avec celles de $\M_{g,[m]}[G]/\Aut(G)$.
	\end{cor}
	
	\bigskip
	
	Étant donné une courbe $\mathcal X/\CC$, \emph{via} le choix d'un difféomorphisme de $\mathcal X$ avec la surface de référence de genre $g$ et $m$-marquée, on définit un homomorphisme injectif $\theta\colon\Aut_\CC(\mathcal X) \to \Gamma_{g, [m]}$. Le choix d'un autre difféomorphisme s'obtient en conjuguant l'image de $\theta$ dans $\Gamma_{g, [m]}$. Par suite, si $\mathcal G$ est la classe de conjugaison dans $\Gamma_{g, [m]}$ d'un sous-groupe isomorphe à $G$, alors la condition $\Aut_\CC(\mathcal X)\in \mathcal G$ est indépendante de tous les choix.
	
	\begin{defi}\label{defi:lieuSpec}
		On nomme lieu spécial de $\mathcal{G}$, et on note $\M_{g,[m],\mathcal{G}}$, la sous-catégorie de $\M_{g,[m]}(G)$ fibrée en groupoïde sur $Sch/\mathbb{C}$ classifiant les objets dont les fibres $\X_s$ vérifient $\Aut(\X_s)\in \mathcal{G}$.
	\end{defi}
	Notant $\mathcal Z$ une composante irréductible de $(\M_{g,[m]}[G]/\Aut(G))_{\mathbb{C}}$, la discussion précédente se résume alors par le diagramme suivant :
	\[
	\xymatrix{
		\mathcal Z\ar[d]\ar[r]&(\M_{g,[m]}[G]/\Aut(G))_{\mathbb{C}}\ar[d] & \\
		\M_{g,[m],\mathcal{G}}\ar[r]&\M_{g,[m]}(G)_ {\CC}\ar[r]&\left(\M_{g,[m]}\right)_\mathbb{C}
	} 
	\]

	L'espace $\M_{g,[m],\mathcal{G}}$ possède une structure de champ algébrique par le résultat ci-après qui établit une correspondance bijective entre l'ensemble des classes de conjugaison de $G$ dans $\Gamma_{g,[m]}$ et l'ensemble des composantes irréductibles du lieu spécial $\M_{g,[m]}(G)$.

	\begin{prop}\label{prop:irrIsconj}
		Soit $G$ un sous-groupe fini du mapping class group, et $\mathcal{G}$ sa classe de conjugaison dans $\Gamma_{g,[m]}$. Alors le lieu spécial $\M_{g,[m],\mathcal{G}}$ est un sous-champ algébrique, composante irréductible du champ des lieux spéciaux $\M_{g,[m]}(G)_{\CC}\subset \left(\M_{g,[m]}\right)_\CC$.
	\end{prop}
	Ce résultat d'irréductibilité est établi par S.~Broughton \cite{BRO90} pour les lieux spéciaux d'espace des modules sans points marqués $\M_g$ et repose sur un théorème de Riemann généralisé de D.~Mumford \cite{MUM67} de type GAGA. Il est étendu aux courbes marquées $\M_{g,[m]}$ dans \cite{Cui08}.

	\subsection{Composantes irréductibles, actions galoisiennes}
	Soit $G$ un sous-groupe fini de $\Gamma_{g,[m]}$ et $\cal G$ une classe de conjugaison de $G$ dans $\Gamma_{g,[m]}$ fixée. L'étude de l'action de $\Gq$ sur $\cal G$ passe par celle de l'ensemble des composantes irréductibles du normalisé $\cal M_{g,[m]}[G]/\Aut(G)$ de $\cal M_{g,[m]}(G)$ de la fa\c con suivante.
	
	\bigskip
	
	Suivant \cite{ODA}, le choix d'un point base géométrique $x\in (\M_{g,[m]})_{\bar{\QQ}}$ ainsi que d'un $\QQ$-point de $\M_{g,[m]}$ définit une action du groupe de Galois absolu :
	\begin{equation}\label{eq:SFE}
		\Gal(\overline{\QQ}/\QQ)\to \Aut[\pi_1^{alg}((\M_{g,[m]})_{\overline{\QQ}},x)].
	\end{equation}

	Le groupe de Galois $\Gq$ agit alors sur la classe de conjugaison $\cal G$ de $G$ dans $\Gamma_{g,[m]}$ définie par la composante irréductible $\cal M_{g,[m],\cal G}$ de $\M_{g,[m]}(G)_{\bar{\QQ}}$ à laquelle appartient $x$. En effet, pour tout $y\in \M_{g,[m],\cal G}$
	\[
	I_y\hookrightarrow \pi_1^{alg}((\M_{g,[m]})_{\bar{\QQ}},x)
	\]
	est donné par conjugaison dans $\Gamma_{g,[m]}$ suivant \cite[Rem. 3.5]{NOOHI04}.

	Dans le cas où on ne dispose pas d'un point base rationnel, on travaille avec un \emph{point base tangentiel} défini comme foncteur fibre sur la catégorie $\text{FEt}^D(\overline{\M}_{g,[m]})$ des revêtements étales finis de $\overline{\M}_{g,[m]}$ modérément ramifiés au dessus d'un diviseur à croisement normaux $D$, où  $\overline{\M}_{g,[m]}$ est la compactification de Deligne-Mumford et $D$ le complémentaire de $\M_{g,[m]}$ dans celle-ci. Ce foncteur est isomorphe à un foncteur fibre défini par un point géométrique de $\M_{g,[m]}$ et l'on dispose ainsi des notions de chemins entre points tangentiels ou géométriques, ainsi que de l'invariance du groupe fondamental par changement de corps algébriquement clos en caractéristique nulle -- voir \cite[\textsc{II}.2,\textsc{II}.4]{ZOO2001}. 
	
	Suivant \cite{NAKA99}, le choix d'un tel point base tangentiel $\overrightarrow b$ de $\M_{g,[m]}$ définit de plus une action galoisienne telle qu'en \eqref{eq:SFE} :
	\begin{equation}
		\textrm{Gal}(\overline{\QQ}/\QQ)\to \Aut[\pi_1^{alg}((\M_{g,[m]})_{\overline{\QQ}},\overrightarrow b)],
	\end{equation}
	où $\Gq$ agit à la fois sur l'ensemble des classes de conjugaison de $G$ (resp. des composantes irréductibles de $\M_{g,[m]}(G)_{\bar{\QQ}}$), ainsi que sur chacune de ces classes (resp. chacune des $\M_{g,[m],\cal G}(G)$). Le calcul de cette dernière action, qui constitue l'objectif principal de cet article, est menée du point de vue de l'action sur les composantes irréductibles de $\M_{g,[m]}(G)$. Ainsi, la stabilité de chacune des classes de conjugaison d'un groupe cyclique fini de $\pi_1^{geom}(\M_{g,[m]})$ sous l'action de $\Gq$ se réduit à la stabilité de chacune des composantes irréductibles du champ des lieux spéciaux $\M_{g,[m]}(G)_{\overline \QQ}$ d'après la proposition \ref{prop:irrIsconj}.
	
	\bigskip
	
	De ce point de vue, le groupe $\Gq$ agit par permutation sur l'ensemble des composantes irréductibles $\{\cal M_{g,[m],\cal G}\}_{\cal G}$ de $\cal M_{g,[m]}(G)$ (resp. l'ensemble des classes de conjugaison de $G$ dans $\Gamma_{g,[m]}$): pour $\sigma\in\Gq$, une composante $\cal M_{g,[m],\cal G}$ est envoyé sur une composante $\cal M_{g,[m],\cal G}^{\sigma}$. Par fonctorialité, cette action est compatible avec l'action définie précédemment sur le groupe fondamental. Le résultat suivant découle de la proposition \ref{prop:norm}.
	\begin{prop}\label{prop:galCompIrr}
		Le groupe de Galois absolu $\Gq$ agit de façon compatible sur l'ensemble des composantes irréductibles de $(\M_{g,[m]}[G]/\Aut(G))_{\overline \QQ}$ et de $\M_{g,[m]}(G)_{\overline \QQ}$.
	\end{prop}
	
	La question initiale de la stabilité d'une classe de conjugaison se réduit ainsi à celle des composantes irréductibles du champ $(\M_{g,[m]}[G]/\Aut(G))_{\overline \QQ}$ selon les propositions \ref{prop:irrSpecIrrHur} et \ref{prop:galCompIrr}.

	Une caractérisation des composantes irréductibles du champ $(\M_{g,[m]}[G]/\Aut(G))_{\overline \QQ}$ étant donnée en sections \ref{sec:brData} et \ref{sec:actionGal}, on établit qu'elles sont chacune définies sur $\QQ$ et donc leur invariance sous l'action de $\Gq$. Fixant une telle composante, on construit alors en section \ref{sec:actionGal} un $K$-point pour lequel on montre que le calcul de l'action de $\Gq$ sur le groupe d'inertie se réduit à celui de l'action de $\Gal(\bar K/K)$.

	\begin{rem}
		Le théorème 3.10 de \cite{Colg0} et le théorème 6.5 de \cite{Colg1} sur les classes de conjugaison des mapping class groups impliquent le théorème A respectivement en genre zéro et un et pour la torsion \emph{géométrique d'ordre premier}. 
	\end{rem}

	\section{Espaces des modules des courbes avec $G$-automorphisme, invariants de revêtements}\label{sec:brData}
	Soit $G$ un groupe cyclique fini. Le champs des courbes avec $G$-automorphisme $\mathcal M_{g, [m]}[G]$ n'étant en général pas irréductible, nous introduisons des sous-champs $\M_{g,[m],\underline{\bf kr}}[G]$ de $\mathcal M_{g, [m]}[G]$ à partir desquels nous caractérisons les composantes irréductibles de $\M_{g,[m]}[G]/\Aut(G)$ au paragraphe \ref{Irreductibilite}. 
	
	La construction de ces espaces repose sur les \emph{données de Hurwitz} associées à un $G$-revêtement cyclique ainsi que sur des \emph{données de branchement} dont nous donnons une construction algébrique dans le cas des espaces de modules de courbes \emph{à points marqués} et définis sur un corps algébriquement clos \emph{quelconque}.

	Cette approche nous permet de résoudre des questions de rationnalité dans le cas d'un corps quelconque (voir section \ref{Rationnalite}), étape nécessaire pour établir l'irréductibilité des espaces  $\M_{g,[m],\underline{\bf kr}}[G]/\Aut(G)$ dans la section \ref{Irreductibilite}.
	
	\subsection{Données de Hurwitz, une construction algébrique}
	Dans \cite{ COR87}, F.~Cornalba donne une construction des données de Hurwitz d'un $G$-revêtement pour les espaces de modules de courbes non-marquées et d'un point de vue complexe. La construction algébrique que nous donnons ici étend cette dernière au cas de points marqués et d'un corps de définition algébriquement clos quelconque. 
	Elle se place dans un contexte de cohomologie étale et repose sur une description géométrique bien connue du lieu de branchement d'un revêtement - voir par exemple \cite{BERO07} section 3.1. 
	
	\medskip
	
	Dans la suite, nous considérons $\mathcal X/S$ une courbe propre et lisse, munie d'une action $\iota\colon G \to \Aut_S(\mathcal X)$ d'un groupe cyclique d'ordre inversible sur $S$, et noterons $B \subset \mathcal Y=\mathcal X/G$ le diviseur de branchement.
	\begin{lem}\label{lem:nuCst}
		Soit $\mathcal X/S$ une courbe telle que ci-dessus. Alors la fonction $\nu$ comptant le nombre de points géométriques de branchement dans les fibres
		\begin{equation*}
			\begin{array}{cccc}
				\nu \colon & S & \to & \NN \\
				& s & \mapsto & \card \left(( B\times_S s)(\bar s)\right)
			\end{array} 
		\end{equation*}
		est localement constante.
	\end{lem}
	
	La constance locale de la fonction $\nu$ montre que, étant donnée un entier $\nu_0\in \NN$, 
	le lieu de points de $\mathcal M_{g, [m]}$ où la fonction $\nu$ vaut $\nu_0$ est un sous- champs algébrique qu'on 
	notera $\mathcal M_{g, [m], \nu_0}$.

	\begin{proof}
		Ceci découle directement de \cite[Lemme 3.3]{BERO07}. En effet, puisque $G$ est d'ordre inversible sur $\O_S$, pour tout $H\subset G$ le schéma $\mathcal X^H$ des points fixes sous $H$ est naturellement un diviseur de Cartier relatif étale sur $S$. Il en est ainsi de même pour le support de $B$ qui est égal à l'image de $\cup_{\{0\} \not = H \subset G} \mathcal X^H$ dans $\mathcal Y$, d'où la constance de $\nu$.
	\end{proof} 
	
	C'est le fait que les groupes considérés ont des ordres premiers aux caractéristiques résiduelles qui impose cette géométrie du lieu de branchement particulièrement simple.
	
	\medskip
	
	Considérons une courbe $G$-équivariante $\X$ et notons $\Y=\X/G$, $R$ le diviseur de ramification, $B$ le diviseur de branchement et $\pi\colon (\mathcal X\setminus R) \to \mathcal Y\setminus B$ le morphisme quotient, ainsi que le diagramme
	\begin{equation}
		\label{DiagrammeOuvert}
		\xymatrix{
			\Y\setminus B \ar@{^(->}[r]^j \ar[rd]_f &  \Y\ar[d]^h \\
			& S.
		}
	\end{equation}
	Suivant la correspondance entre classes d'isomorphismes de $G$-torseurs sur $\Y\setminus B$ et éléments de $\H^1_\et( \Y\setminus B, G)$, le morphisme étale $\pi$ définit alors un élément de $\H^1_\et( \Y\setminus B, G)$. Ceci motive la définition suivante.
	
	\begin{defi}\label{defi:HurwData}
		Soit $(\mathcal{X},\iota)$ une courbe propre lisse munie d'une $G$ action définissant un élément $\pi\in \H^1_\et( \Y\setminus B, G)$ comme ci-dessus. On nomme donnée de Hurwitz associée à $(\mathcal{X},\iota)$ l'image de $\pi$ dans $\H^0_\et(S, h_* \R^1 j_* G)$ par les homomorphismes 
		\begin{equation*}
			\H^1_\et( \mathcal Y\setminus B, G) \to \H^0_\et(S, \R^1 f_* G) \to \H^0_\et(S, h_* \R^1 j_* G).
		\end{equation*}
	\end{defi}

	\medskip
	
	Décrivont explicitement les fibres du faisceau $h_* \R^1 j_* G$ à travers $\H^0_\et(\mathcal Y, \R^1 j_* G)$, puisque selon le théorème de changement de base propre $(h_* \R^1 j_* G)_s\simeq\H^0_\et(\mathcal Y_s, \R^1 j_* G)$ pour $s\in S$.
	
	\begin{lem}\label{lem:faisc}
		Avec les notations de \eqref{DiagrammeOuvert} ci-dessus, le faisceau $h_* \R^1 j_* G$ est localement constant et la fibre en un point géométrique $s \in S$ est isomorphe à $\bigoplus_{y \in B_s} \left(\ZZ/n\ZZ\right)$.
	\end{lem}
	Notons que le théorème de changement de base lisse en cohomologie étale implique que le faisceau $h_* \R^1 j_* G$ est constructible. Le résultat découle ici d'une description explicite du faisceau.
	
	\begin{proof}
		Considérons le spectre $S$ d'un corps algébriquement clos et fixons un isomorphisme $G \isomto \mmu_n$. On montre alors que la suite exacte de Kummer donne une suite exacte
		\[
		1\to j_*\mmu_n\to j_*\GG_m\to j_*\GG_m\to \R^1j_*\mmu_n\to 1,
		\]
		la surjectivité de la dernière flèche provenant de \cite{EGA4}[21.6.9] par lissité de la courbe $\cal Y$.
		D'autre part, le morphisme 
		\begin{equation*}
			\begin{array}{ccc}
				j_* \GG_m & \to & \displaystyle\bigoplus_{y \in B} i_{y*} \ZZ/n\ZZ \\
				\alpha & \mapsto & (v_y(\alpha))_y
			\end{array} 
		\end{equation*}
		où $i_y$ désigne l'inclusion $y \to \mathcal Y$ et $v_y$ la valuation normalisée sur l'anneau local $\O_{\mathcal Y, y}$, est nul sur les puissances $n$-ième. Ceci induit un isomorphisme
		\begin{equation*}
			\R^1 j_* \mmu_n \isomto  \displaystyle\bigoplus_{y \in B} i_{y*} \ZZ/n\ZZ \\
		\end{equation*}
		puis l'isomorphisme recherché
		\begin{equation}
			\label{IsomHurwitz}
			\H^0_\et(\mathcal Y, \R^1 j_* G) \isomto \bigoplus_{y \in B_s} \left(\ZZ/n\ZZ\right).
		\end{equation}
	\end{proof}
	
	\begin{rem}\label{rem:nonCan}
		Il est important de noter \emph{la non-canonicité} de l'isomorphisme \eqref{IsomHurwitz}. En effet, les changements d'isomorphisme $G \isomto \mmu_n$ induisent par fonctorialité une action de $\Aut(G)$ sur le membre de gauche de l'équation \ref{IsomHurwitz}\footnote{Les deux auteurs remercient chaleureusement le referee pour son insistance sur ce point.}.
	\end{rem}

	\begin{defi}\label{defi:DataHurwitz}
		Soit $G$ un groupe cyclique d'ordre $n$.
		Une donnée de Hurwitz géométrique abstraite $ {\bf k}$ pour $G$ à $\nu$ points de branchement est un élément
		$ {\bf k} \in \left(\ZZ/n\ZZ\right)^\nu/\mathfrak S_\nu$, tel que la somme des coordonnées est nulle -- l'action de $ \mathfrak S_\nu$ se faisant par permutation des coordonnées.
	\end{defi}
	
	Étant donné un $G$-revêtement et un isomorphisme $G\isomto\mmu_n$ fixé, pour tout point géométrique $s\in S$, on associe un élément de $\H^0_\et(\mathcal Y_s, \R^1 j_* G)$ suivant la définition \ref{defi:HurwData}, et qui s'identifie canoniquement à $\bigoplus_{y \in B_s} \left(\ZZ/n\ZZ\right)$ d'après le lemme \ref{lem:faisc}. Le quotient par $\mathfrak{S}_{\nu}$ fournit alors une donnée de Hurwitz abstraite. 
	
	\begin{rem}\label{rem:ConstAndCano}\mbox{}
		\begin{enumerate}[label=\roman*),ref=\roman*)] 
			\item \emph{Composantes connexes --}\label{ConstanceLocaleDonneeHurwitz}
			Étant donnés un corps algébriquement clos $k$, un isomorphisme $G  \isomto \mmu_n$ et une donnée de Hurwitz géométrique abstraite à $\nu$ points de branchements, on définit, pour tout $k$-schéma $S$ le lieu des points $s$ de $S$ tel que $(\mathcal X, \iota)$ possède $ {\bf k}$ comme donnée de Hurwitz géométrique abstraite en $s$. La constance locale du faisceau $h_* \R^1 j_* G$ implique que celui-ci est ouvert et fermé dans $S$ : c'est une réunion de composantes connexes de $S$.
			
			\item\label{rem:choixCano} \emph{Données canoniques --}
			Un $\mmu_n$-torseur étant localement donné par $w^n=\alpha$, le membre de droite de l'équation \eqref{IsomHurwitz} est constitué de l'ordre des zéros et des pôles de $\alpha$ en chaque point de $B$ modulo $n$. L'action de $\Aut(G)$ y est donnée par multiplication sur $\div(\alpha)$ après l'identification canonique $\Aut(G)=(\ZZ/n\ZZ)^*$. L'indépendance du choix d'un isomorphisme $\bigoplus_{y \in B_s} \ZZ/n\ZZ \isomto \left(\ZZ/n\ZZ\right)^\nu$ est alors donné par double passage au quotient, par $\mathfrak S_\nu$ et $(\ZZ/n\ZZ)^*$.
		\end{enumerate}
	\end{rem}
	
	\medskip
	
	Réciproquement, pour une donnée de Hurwitz abstraite $\bf  k$ et un isomorphisme $G\simeq \mmu_n$ fixés, la proposition suivante produit un $G$-revêtement de données de Hurwitz $\bf k$.
	
	\begin{prop}
		Soit $\Y$ courbe propre et lisse sur $k$ algébriquement clos, $G$ un groupe fini cyclique, $G\simeq \mmu_n$ un isomorphisme fixé,  et $\bf k$ une donnée de Hurwitz géométrique abstraite pour $G$. Alors il existe un revêtement galoisien $\X/\Y$ de groupe $G$ et de donnée de Hurwitz $\bf k$.
	\end{prop}
	
	\begin{proof}
		Le morphisme $\R^1 f_* G \to h_* \R^1 j_* G$ déjà utilisé ci-dessus s'inscrit dans une suite exacte 
		longue donnée par la suite spectrale de Leray :
		\begin{equation*}
			0 \to \R^1 h_* (j_* G) \to \R^1 f_* G \to h_* \R^1 j_* G \to \R^2 h_* (j_* G) \to \R^2 f_* G. 
		\end{equation*}

		Si $\nu = 0$ alors la donnée de Hurwitz est vide, on peut donc supposer $\nu > 0$. 
		Ainsi $f$ est affine et donc $\R^2 f_* G=0$. 
		Comme $j_* G\cong G$ et que tout les faisceaux de cette suite sont finis sur $S$, si ce dernier est 
		le spectre d'un corps algébriquement clos on trouve une suite exacte
		\begin{equation}
			\label{LocalGlobalLeray}
			\H^0_\et(S, \R^1 f_* G) \to \H^0_\et(S, h_* \R^1 j_* G) \to \H^0_\et(S, \R^2 h_* (j_* G)) \to 0
		\end{equation}
		la dernière flèche non nulle s'identifiant à
		\begin{equation*}
			\begin{array}{ccc}
				(\ZZ/n\ZZ)^\nu & \to & \ZZ/n\ZZ \\
				(k_i)_i & \mapsto & \sum_i k_i.
			\end{array}
		\end{equation*}

		En particulier, étant donné ${\bf k}$ dans le noyau de cette application, il existe 
		$a \in \H^0_\et(S, \R^1 f_* G)$ qui est envoyé sur ${\bf k}$. Or, comme $S$ est strictement hensélien, on a 
		\begin{equation*}
			\H^0_\et(S, \R^1 f_* G)=\H^1_\et(\mathcal Y\setminus B, G). 
		\end{equation*}

		À l'élément $a$ correspond donc un $G$-torseur au-dessus de $\mathcal Y\setminus B$ qui fournit le revêtement $\X/\Y$ cherché.
	\end{proof}
	
	\begin{rem}
		Au dessus d'un corps $k$ algébriquement clos, la théorie de Kummer donne une description explicite des données de Hurwitz, puisqu'un $G$-revêtement $\X \to \Y$ de diviseur de branchement $B$ est localement donné par une équation $w^n=\alpha$, où $\alpha$ est une fonction définie sur un ouvert de $\Y \setminus B$. Pour $y \in B$ la donnée de Hurwitz du revêtement $X \to Y$ en $y$ est l'ordre d'annulation de $\alpha$ en $y$ modulo $n$ comme dans la preuve du lemme 3.3 (en particulier, la ramification est incluse dans le support de $\textrm{div}(\alpha)$ mais ne lui est pas égale en général).

		On retrouve en particulier le résultat de classification des $G$-revêtements établi par F.~Cornalba sur $\mathbb C$.
	\end{rem}

	\subsection{Un sous-champ algébrique de $\M_{g,[m]}[G]$}
	\label{EspaceModule}
	
	On donne ici la construction de sous-champs $\mathcal M_{g, [m], \underline {\bf kr}}[G]$ de $\M_{g,[m]}[G]$ à partir desquels nous caractérisons les composantes irréductibles de $ \M_{g,[m]}[G]$ dans la section suivante. Ces sous-champs sont construits à partir de données de Hurwitz géométrique ${\bf k}$ de la section précédente, et de branchement $\bf kr$ définie ci-dessous.
	
	\medskip
	
	Dans la suite, on identifie $\NN^n$ à $\NN^{\ZZ/n\ZZ}$ afin de le munir de l'action naturelle de $\Aut(G)$ canoniquement identifié à $(\ZZ/n\ZZ)^*$.
	
	\begin{defi} 
		Soit l'application 
		\[{\bf kr} \colon \mathcal M_{g, [m], \nu_0}[G] \longrightarrow \left(\ZZ/n\ZZ\right)^{\nu_0}/\mathfrak{S}_{\nu_0}\times\NN^{n}
		\] 
		qui à :
		\begin{itemize}
			\item un point géométrique $s$, i.e. un couple $(\mathcal X, \iota)$ muni d'un diviseur équivariant $D$ de degré $m$ ;
			\item un isomorphisme $G\simeq \mmu_n$ ;
		\end{itemize}
		associe le couple $({\bf k}(s),{\bf r}(s))$ où ${\bf k}(s)$ est une donnée de Hurwitz géométrique et ${\bf r}(s)$ un $n$-uplet de $i$-ème coordonnée  
		\begin{equation*}
			{\bf r}(s)_i=\#\{y \in D_{s}/G,\ br(y)\text{ est égal à } i \mod n\}, 
		\end{equation*}
		où $br(y)$ désigne l'ordre de branchement en $y$ -- l'ordre de branchement est dit nul modulo $n$ si le morphisme est étale au-dessus du point considéré. 
		
		On nomme \emph{donnée de branchement} l'image $\underline{ \bf kr}={\bf kr} \mod (\ZZ/n\ZZ)^*$ sous l'action diagonale.
	\end{defi}
	
	On rappelle que l'action de $(\ZZ/n\ZZ)^*$ sur le facteur $\left(\ZZ/n\ZZ\right)^{\nu_0}$ est donnée par 
	la multiplication.

	Notons que la donnée de branchement $\underline{\bf kr}$ est indépendante du choix d'un ismorphisme $G\simeq \mmu_n$ -- voir remarque \ref{rem:ConstAndCano} - \ref{rem:choixCano}. Les sous-champs qui nous occupent sont alors les suivants.
	\begin{defi}
		\label{DefCompIrred}
		Soient $g, m, g', \nu_0 \in \NN$ des entiers, $G$ un groupe cyclique d'ordre $n$ inversible dans $S$ et $\underline {\bf kr}$ une donnée de branchement abstraite. On note $\mathcal M_{g, [m], \underline {\bf kr}}[G]$ le lieu des points $(\mathcal X/S, \iota, D)$ de $\M_{g,[m]}[G]$ munis d'un diviseur $G$-équivariant étale $D$ de degré $m$, et tel que
		\begin{enumerate}[label=\roman*),ref=\roman*)]
			\item $\mathcal Y=\mathcal X/G$ est une courbe propre et lisse de genre $g'$ ;
			\item en tout point géométrique $s \to S$ le morphisme $\mathcal X_{s} \to \mathcal Y_{s}$
			est ramifié au-dessus d'exactement $\nu_0$ points ;
			\item la donnée de branchement est en tout point égale à $\underline {\bf kr}$.
		\end{enumerate}
	\end{defi}
	Notons que $g'$ est omis de la notation $\mathcal M_{g, [m], \underline {\bf kr}}[G]$ puisqu'entièrement déterminé par les autres paramètres.
	
	\bigskip
	
	Les définitions de $\bf k$ et $\bf r$ qui composent la donnée de branchement $\underline{\bf kr}$ s'illustrent explicitement dans le cas des courbes de quotient de genre $0$.
	\begin{exem}
		Soit $G$ cyclique d'ordre $n$ et considérons $\X \in \cal M_{g,[m]}[G]$, courbe de genre $g$ munie d'une action de $G$ et d'un diviseur $G$-équivariant $D$, dont on suppose que le quotient $\X/G$ est de genre $g'=0$. Notant $\pi\colon\X\rightarrow \mathbb P^1$ le $G$-torseur associé, $\pi$ est donnée par une équation $w^n=\alpha$ entre paramètres. 
		
		Supposons fixé un isomorphisme $G\simeq \mmu_n$. Afin de préciser la définition, supposons que $n\neq 3$ et $n\neq 2$, que $\alpha$ se factorise en $\alpha(x)=(x-1)^3(x+1)^2$ sur $\bar{\QQ}$, et que $D$ est donné par le marquage des points $\pi^{-1}(\{1,2\})$ dans $\X$. Alors, le point $1$ (resp. $2$) de $D/G=\{1,2\}$ donne $r_3=1$ (resp. $r_0=1$) tandis que $r_i=0$ pour $i\notin \{0,3\}$ dans la donnée $\bf r$, et le revêtement étant branché au dessus de $\{1,-1,\infty\}$, on a $\nu_0=3$ avec $\mathbf{k}=(3,2,-5)\in(\ZZ/n\ZZ)^2$.
	\end{exem}
	
	\begin{prop}\label{prop:irrIsStack}
		Sous les conditions de genre, d'ordre et de donnée de branchement, l'espace $\M_{g,[m],\underline{\bf kr}}[G] $ est un sous-champ algébrique de $\M_{g,[m]}[G]$ défini sur $\QQ$, réunion de composantes connexes de $\M_{g,[m]}[G]$. 
	\end{prop}
	
	La structure de champ de $\mathcal M_{g, [m], \underline {\bf kr}}[G]$ découle en partie du lemme suivant.
	\begin{lem}\label{r_constante}
		La fonction $\underline{\bf kr}$ est localement constante.
	\end{lem}
	
	\begin{proof}
		Fixons un isomorphisme $G\simeq \mmu_n\otimes\bar\QQ$. Il suffit de montrer que la fonction ${\bf kr=(k,r)}\colon \mathcal M_{g, [m], \nu_0}\otimes\bar\QQ \to (\ZZ/n\ZZ)^{\nu_0}\times\NN^n$ est localement constante, c'est-à-dire que  $\bf r$ l'est, puisque la fonction $\bf k$ est localement constante suivant la remarque \ref{rem:ConstAndCano} - \ref{ConstanceLocaleDonneeHurwitz}.	
		
		Concernant $\bf r$, comme $D$ est un diviseur étale, quitte à faire un changement de base étale on peut supposer que $D$ est la réunion de $m$ sections disjointes. Il suffit alors de montrer que le stabilisateur de l'une quelconque de ces sections est constant car la donnée $\bf r$ l'est alors automatiquement. Considérons pour cela un point $x \in D$ et une de ses spécialisations $x'$.
		
		L'action de $G$ sur la courbe étant modérée, on a $G_x \subset G_{x'}$. Il s'agit alors de montrer l'égalité, mais si $g\in G_{x'} \setminus G_x$ alors $x$ et $gx$ ont même spécialisation $x'$, ce qui est absurde car $D$ est étale. 
	\end{proof}

	\begin{proof}[Démonstration la proposition \ref{prop:irrIsStack}]
		La structure de champ de l'espace $\mathcal M_{g, [m], \underline {\bf kr}}[G]$ provient de sa construction comme réunion de composantes connexes de champ. En effet, considérons le lieu $\mathcal M_{g, [m], g'}[G]$ des points de $\mathcal M_{g, [m]}[G]$ constitué des points $(\mathcal X, \iota)$ tel que le quotient $\mathcal X/G$ est une courbe propre et lisse de genre $g'$. Alors suivant \cite{MAU06}, $\mathcal M_{g, [m], g'}[G]$ est une réunion de composantes connexes de $\mathcal M_{g, [m]}[G]$.

		Considérons de plus le lieu $\mathcal M_{g, [m], g', \nu_0}[G]$ des points de $\mathcal M_{g, [m], g'}[G]$ tels que le revêtement $\mathcal X \to \mathcal X/G$ est ramifié au-dessus d'exactement $\nu_0$ points géométriques. Suivant le lemme \ref{lem:nuCst}, la fonction $\nu$ définie au lemme \ref{lem:nuCst} est localement constante, et $\mathcal M_{g, [m], g', \nu_0}[G]$ est un champ algébrique comme réunion de composantes connexes de $\mathcal M_{g, [m]}[G]$.

		Finalement le champ $\M_{g,[m],\underline{\bf kr}}[G]$ étant défini comme le lieu des points de $\mathcal M_{g, [m], g', \nu_0}[G]$, composé des courbes de donnée de branchement égale à $\underline{\bf kr}$, on conclut de même que précédement à partir du lemme \ref{r_constante}.
	\end{proof}
	
	Les champs $\mathcal M_{g, [m], \underline {\bf kr}}[G]$ sont naturellement munis d'une action libre de $\Aut(G)$. Le champ quotient $\mathcal M_{g, [m], \underline {\bf kr}}[G]/\Aut(G)$, dont les composantes irréductibles sont en bijection avec celles du lieu spécial $\M(G)$, classifie alors les courbes dont le groupe d'automorphisme contient un sous-groupe isomorphe à $G$.

	\section{Composantes irréductibles de l'espace des modules des courbes avec $G$-automorphisme}\label{sec:CompIrr}
	Dans cette section, nous établissons l'irréductibilité des sous-champs $\M_{g,[m],\underline{\bf kr}}[G]/\Aut(G)$ du champ $\M_{g,[m]}[G]/\Aut(G)$. Ce résultat s'obtient en se ramenant au cas déjà connu pour $m=0$ et en montrant l'existence, pour une courbe $\Y/K$ propre et lisse sur un corps, et pour une donnée de Hurwitz fixée, de l'existence d'un $G$-revêtement de $\Y$ de même donnée de Hurwitz et stable sous l'action du groupe de Galois $\Gal(\overline K/K)$. Cette étape nous permet de décrire dans la section \ref{Irreductibilite} l'action galoisienne sur le lieu de branchement d'un revêtement générique.
	
	\subsection{Rationnalité et données de Hurwitz}
	\label{Rationnalite}
	
	Considérons en toute généralité une immersion ouverte $j\colon \U \hookrightarrow \Y$ dans une courbe propre et lisse $\Y/K$, ainsi que ${\bf k} \in \H^0(\Y_{\bar K}, \R^1 j_* \mmu_n) \cong (\ZZ/n\ZZ)^\nu$ fixée et dont la somme des coordonnées est nulle, où $\nu$ désigne le nombre de points géométriques de $\Y \setminus \U$.

	Supposons de plus que ${\bf k}$ est stable sous l'action du groupe de Galois $\Gamma=\Gal(\bar K/K)$ induite sur $\H^0_\et(\Y_{\bar K}, \R^1 j_* \mmu_n)$ -- l'action étant donnée par permutation des indices $\{1, \ldots, \nu\}$. On montre ici qu'il existe un revêtement galoisien de $\Y$, de groupe $\ZZ/n\ZZ$, ramifié exactement au-dessus de $\Y\setminus \U$, et dont les données de Hurwitz géométriques correspondantes sont ${\bf k}\in (\ZZ/n\ZZ)^\nu$.

	\begin{theo}\label{DescenteEquivariante}
		Soit $K$ un corps de caractéristique première à $n$ contenant les racines $n$-èmes de l'unité, $\Y/K$ une courbe propre et lisse et $j\colon \U \hookrightarrow \Y$ une immersion ouverte avec $\U$ non vide. Notons $\bar K$ une clôture séparable de $K$, et $\Gamma=\Gal(\bar K/K)$ son groupe de Galois. Si $\Y$ possède un point $K$-rationnel, alors la suite
		\begin{equation*}
			\H^1_\et(\U_{\bar K}, \mmu_n)^\Gamma \to  \H^0_\et(\Y_{\bar K}, \R^1 j_* \mmu_n)^\Gamma \to \H^2_\et(\Y_{\bar K}, \mmu_n) \to 0
		\end{equation*}
		déduite de \eqref{LocalGlobalLeray} en prenant les points fixes sous $\Gamma$ est exacte.
	\end{theo}
	
	\begin{proof}
		Écrivant la suite exacte de Kummer sur $\U_{\bar K}$
		\begin{multline*}
			0 \to \mmu_n(K) \to \GG_m(\U_{\bar K}) \to \GG_m(\U_{\bar K}) \to \H^1_\et(\U_{\bar K}, \mmu_n) \to \H^1_\et(\U_{\bar K}, \GG_m) \\
			\to \H^1_\et(\U_{\bar K}, \GG_m)
		\end{multline*}
		on tire un diagramme commutatif de $\Gamma$-modules à lignes et colonnes exactes
		
		\begin{small}
			$$\xymatrix@C=0.3cm{
				& & \GG_m(\U_{\bar K})/\GG_m(\U_{\bar K})^n \ar[d] \ar@{=}[r] &  \GG_m(\U_{\bar K})/\GG_m(\U_{\bar K})^n \ar[d] \\
				0 \ar[r] & \Pic^0(\Y_{\bar K})[n] \ar[r] \ar@{=}[d] & \H^1_\et(\U_{\bar K}, \mmu_n) \ar[r] \ar[d]^{(1)} & \H^0_\et(\Y_{\bar K}, \R^1 j_* \mmu_n) \ar[r] \ar[d]& \H^2_\et(\Y_{\bar K}, \mmu_n) \ar@{=}[d]\\
				0 \ar[r] & \Pic^0(\Y_{\bar K})[n] \ar[r] & \Pic(\U_{\bar K})[n] \ar[r]^<(.1){(2)} \ar[d] & \H^0_\et(\Y_{\bar K}, \R^1 j_* \mmu_n)/\GG_m(\U_{\bar K}) \ar[d]\ar[r] & \H^2_\et(\Y_{\bar K}, \mmu_n) \\
				& & 0 & 0}$$
		\end{small}
		et il suffit de montrer que toutes les suites de ce diagrammes restent exactes par passage aux points fixes sous $\Gamma$.
		
		On voit qu'il suffit en fait de montrer que la ligne du bas est exacte et que la flèche $(1)$ reste surjective.
		
		Pour cela, il convient d'expliciter autant que possible la flèche $(2)$. Notons 
		$\{y_1 \ldots, y_\nu\}=(\Y\setminus \U)(\bar K)$. 
		On a un isomorphisme naturel de $\Gamma$-modules définit \emph{via} la suite de Kummer et les valuations normalisées en les $y_i$
		\begin{equation*}
			\H^0_\et(\Y_{\bar K}, \R^1 j_* \mmu_n) \isomto \bigoplus_{1 \le i \le \nu} \ZZ/n\ZZ,
		\end{equation*}
		l'action de $\Gamma$ sur $\bigoplus_{1 \le i \le \nu} \ZZ/n\ZZ$ se faisant par permutation sur les indices.
		
		D'autre part, toujours en utilisant Kummer, on montre que
		$\H^2_\et(\Y_{\bar K}, \mmu_n)=\ZZ/n\ZZ$, l'action de $\Gamma$ étant triviale sur 
		les deux modules. Le morphisme 
		\begin{equation*}
			\H^0_\et(\Y_{\bar K}, \R^1 j_* \mmu_n)\to\H^2_\et(\Y_{\bar K}, \mmu_n)
		\end{equation*}
		s'identifie ainsi à 
		\begin{equation*}
			\begin{array}{ccc}
				\bigoplus_i \ZZ/n\ZZ & \to & \ZZ/n\ZZ \\
				(k_i)_i & \mapsto & \sum_i k_i.
			\end{array} 
		\end{equation*}

		La flèche $(2)$ est alors définie de la manière suivante : soit $\mathcal L \in \Pic(\U_{\bar K})[n]$, prenant une représentation par des diviseurs de Weil, il existe alors $\mathcal M \in \Pic^0(\Y_{\bar K})$ telle que $\mathcal M |_\U\cong \mathcal L$ . Comme $\mathcal M^n|_\U = \mathcal L^n\cong \O_\U$, 
		il existe $\tilde k_1 \ldots, \tilde k_\nu \in \ZZ$ tels que $\mathcal M^n \cong \O_{\Y_{K^s}}(\sum_i \tilde k_i y_i)$.
		On associe ainsi à $\mathcal L$ l'image de $(\tilde k_i)_i$ dans $(\bigoplus_i \ZZ/n\ZZ)/\GG_m(\U_{\bar K})$.
		Comme $\mathcal M$ est de degré $0$, on a automatiquement $\sum_i \tilde k_i=0$.

		On retrouve de plus la surjectivité aisément. En effet, soit $\tilde k_1, \ldots, \tilde k_\nu \in \ZZ$ 
		tels que $\sum \tilde k_i=0$ et considérons $\mathcal L=\O_{\Y_{\bar K}}(\sum \tilde k_i y_i)$. 
		Alors il existe $\mathcal M \in \Pic^0(\Y_{\bar K})$ tel que $\mathcal M^n=\mathcal L$ et on a $\mathcal M^n|_\U \cong \O_\U$.
		
		\bigskip
		
		Supposons désormais qu'il existe un point rationnel $y$ dans $\Y$ et considérons 
		une famille $(\tilde k_i)_i$ fixe sous $\Gamma$ (pour la permutation des indices) et 
		telle que $\sum_i \tilde k_i = 0$.
		
		Pour tout $\alpha\in \{1, \ldots, \nu\}/\Gamma$, choisissons un représentant $i \in \alpha$ et fixons un
		faisceau inversible $\mathcal L_\alpha$ racine $n$-ième de $\O_{\Y_{\bar K}}(y_i-y)$. En particulier, pour tout 
		$\sigma \in \Gamma$ on a $\sigma \mathcal L_\alpha^n=\O_{\Y_{\bar K}}(\sigma(y_i)-y)$ car $y$, vu comme point 
		de $\Y(\bar K)$, est fixe sous $\Gamma$.
		
		Notons de plus $\Gamma_\alpha$ l'image dans $\mathfrak S_\nu$ du stabilisateur de $y_i$ dans $\Gamma$, 
		$\tilde k_\alpha=\tilde k_i$ et 
		$$\mathcal L=\bigotimes_{\left(\alpha\in\{1, \ldots n\}/\Gamma\right)} \bigotimes_{\sigma \in \Gamma_\alpha} \sigma \mathcal L_{\alpha}^{\tilde k_\alpha}.$$

		Alors, par construction, $\mathcal L$ est fixe sous $\Gamma$ et on a 
		$\mathcal L^n=\bigotimes_i \O_{\Y_{\bar K}}(\tilde k_i y_i-\tilde k_i y)=\O_{\Y_{\bar K}}(\sum_i \tilde k_i y_i)$. 
		Ce qui prouve l'exactitude du diagramme en la source de $(2)$.
		
		\bigskip
		
		Reste à voir la surjectivité de $(1)$ : considérons un faisceau inversible $\mathcal L$ fixe sous $\Gamma$ dont 
		la puissance $n$-ième est isomorphe à $\O_\U$. Puisque $\mathcal L \in (\Pic(\U_{\bar K})[n])^\Gamma=\Pic(\U)[n]$, on peut choisir un isomorphisme $\Gamma$-équivariant $\mathcal L^n \to \O_\U$ et on construit ainsi de manière $\Gamma$-équivariante un $\mmu_n$-torseur comme dans \cite{COR87}.
	\end{proof}

	Pour une donnée de Hurwitz géométrique invariante sous $\Gamma$ fixée, on dispose ainsi d'un $\mmu_n$-revêtement de \emph{corps de modules} $K$. Sous condition d'existence d'un point rationnel, ce dernier s'avère être un \emph{corps de définition} du revêtement.
	\begin{cor}
		\label{ExisteneDonneeRationnelle}
		Soient $K$ un corps de caractéristique première à $n$ contenant les racines $n$-ièmes de l'unité, 
		$\bar K$ sa clôture algébrique et $\Gamma=\Gal(\bar K/K)$.
		
		Soit $\Y$ une courbe propre et lisse sur un corps $K$, $y_1, \ldots, y_\nu\in \Y(\bar K)$ des points fermés permutés sous l'action de $\Gamma$ et ${\bf k}=(k_1, \ldots , k_\nu) \in (\ZZ/n\ZZ)^\nu$. Supposons que $\forall \sigma \in \Gamma$, si $\sigma(y_i)=y_j$ alors $k_i=k_j$. 
		
		Alors, si $\Y$ possède un point rationnel distinct des $y_i$, il existe un revêtement galoisien $\X \to \Y$ de groupe $\mmu_n$ ramifié en les seuls points $y_i$ et dont les données de Hurwitz sont $\bf k$.
	\end{cor}
	
	\begin{proof} Les hypothèse définissent un élément de $\H^0_\et(\Y_{\bar K}, \R^1 j_* \mmu_n)^\Gamma$ qui provient de $\H^1_\et(\U_{\bar K}, \mmu_n)^\Gamma$ d'après le théorème \ref{DescenteEquivariante}. Autrement dit, il existe un revêtement galoisien $\X' \to \Y_{\bar K}$ dont le corps de module est $K$. Utilisant \cite{DEB90}, corollaire 3.4, ce corps de modules est aussi le corps de définition puisque 
		$G$ est abélien et $\Y$ possède un point rationnel par hypothèse, ce qui assure la condition (Seq/Split) de \emph{loc. cit}. 
		Le revêtement $\X' \to \Y_{\bar K}$ provient donc d'un revêtement $\X \to \Y$ qui répond à la question. \end{proof}
	
	Ce résultat nous est utile pour établir l'irréductibilité géométrique des champs $\mathcal M_{g, [m], \underline {\bf kr}}[G]/\Aut(G)$ dans la section suivante.
	
	\subsection{Description des composantes irréductibles}
	\label{Irreductibilite}
	Soit $G$ groupe cyclique d'ordre $n$. Nous obtenons ici le premier résultat fondamental de cet article qu'est la description des composantes irréductibles de $\M_{g,[m]}[G]/\Aut(G)$ à partir des données de Hurwitz $\underline{\bf  kr}$.
	\begin{theo}
		\label{theo:DescriptionCompIrred}
		Les champs algébriques $\mathcal M_{g, [m], \underline {\bf kr}}[G]/\Aut(G)$ sont géométriquement irréductibles.
	\end{theo}

	Le résultat suivant précise la forme que peut revêtir l'action galoisienne sur les points  de branchement dans le cas générique : les orbites y sont exactement décrites par les données de branchement. Comme cette propriété est stable par générisation, il suffit de construire un revêtement la satisfaisant.
	
	Pour ${\bf k}\in(\ZZ/n\ZZ)^{\nu}$, on note $r_i({\bf k})$ le nombre de coordonnées dans ${\bf k}$ égales à $i \mod n$. 
	
	\begin{prop}\label{lem:formDiv}
		Soit ${\bf k}\in (\ZZ/n\ZZ)^{\nu}$ des données de Hurwitz géométriques abstraites. Il existe alors un corps $K$ et un morphisme $\Spec K \to \mathcal M_{g, [0], \underline {\bf k,0}}[G]$ tel que 
		\begin{equation*}
			\supp(B)= \coprod_{1 \le i < n} \Spec K_i
		\end{equation*}
		où $B$ est le diviseur de branchement de la courbe correspondante, et où $K_i/K$ est une extension de degré $r_i({\bf k})$ -- resp. vide si ce dernier nombre est nul.
		
		Notant $\tilde K_i$ la clôture galoisienne de $K_i$, les extensions $\tilde K_i/K$ sont alors linéairement disjointes, de degré $r_i( {\bf k})!$ et de groupe de Galois $\mathfrak S_{r_i( {\bf k})}$.
	\end{prop}
	
	\begin{proof} 
		Afin d'appliquer le corollaire \ref{ExisteneDonneeRationnelle}, construisons une courbe $\Y_0/K_0$ qui comporte un point rationnel.
		Considérons le point générique de $\left(\mathcal M_{g', \nu+1}\right)_\CC$. Quitte à ajouter encore des points marqués sur les courbes pour se débarrasser d'une éventuelle gerbe résiduelle en le point générique, on peut supposer que ce point correspond à une courbe $\Y/K$ munie de points $y_0, \ldots, y_\nu$. Pour chaque $i \in \{1, \ldots, n\}$, notons $D_i=\{y_\ell, \ell > 0, k_\ell=i \mod n\}$.
		
		On peut alors, \emph{via} le choix d'un ordre sur les points de $D_i$, définir une action libre et transitive de $\mathfrak S_{r_i( {\bf k})}$ sur $D_i$. Par généricité de $K$, celui-ci acquiert une action de $\Gamma=\prod_i \mathfrak S_{r_i( {\bf k})}$ compatible avec celle sur les $D_i$.
		
		Notons $K_0$ le corps fixé par $\Gamma$ et $\Y_0=\Y/\Gamma$ la courbe quotient. Celle-ci vient naturellement avec les diviseurs $D_i/\Gamma$, irréductibles par construction et de degré $r_i({\bf k})$, le groupe de Galois de leur corps de fonctions sur $K_0$ étant $\mathfrak S_{r_i( {\bf k})}$.

		Notons que la courbe $(\Y_0, \sum_i D_i/\Gamma)$ peut être vue comme la courbe générique de genre $g'$ munie de $n$ diviseurs respectivement de degré $r_i( {\bf k})$ et d'un point rationnel disjoint du support des $D_i/\Gamma$. Le corollaire \ref{ExisteneDonneeRationnelle} fournit alors un revêtement galoisien $\X \to \Y_0$ de groupe $\ZZ/n\ZZ$ possédant les propriétés requises car $\CC \subset K_0$ et donc $K_0$ contient les racines de l'unité.
	\end{proof}

	\begin{proof}[Démonstration du théorème \ref{theo:DescriptionCompIrred}]
		Rappelons que lorsque $m=0$, l'irréductibilité est établie par M.~Cornalba et F.~Catanese respectivement dans \cite{COR87,CAT10}.
		
		Considérons dans un premier temps le cas où les points marqués sont disjoints du lieu de ramification. Notons que la fonction $r_n$ -- donnée par $r_n\colon{\bf kr}=({\bf k,r})\mapsto {\bf r}_n$ -- est invariante sous l'action de $(\ZZ/n\ZZ)^*$ et définit donc une fonction $r_n$ sur $\underline{\bf kr}$. Dans ce cas $r_{n}(\underline{\bf kr})=\frac{m}{n}$, et l'on considère alors l'ouvert $\U$ de $\mathcal M_{g, [0], \underline {\bf (k,0)}}\times_{\mathcal M_{g}} \mathcal M_{g, [\frac{m}{n}]}$ composé des courbes équivariantes $(\mathcal \X, \iota)$ munies d'un diviseur $E$ (\emph{a priori} non équivariant), étale de degré $\frac{m}{n}$ et tel que $hE \cap E = \emptyset$ pour tout $h \in G$ non trivial.
		
		Le morphisme $(\mathcal \X, \iota, E) \mapsto (\mathcal \X, \iota, \coprod_{h \in G} h E)$ définit alors un isomorphisme 
		\begin{equation*}
			\U \isomto \mathcal M_{g, [m], \underline {\bf kr}}[G].
		\end{equation*}
		Comme $\U/\Aut(G)$ est irréductible, car $\mathcal M_{g, [0], \underline {(\bf k,0)}}/\Aut(G)$ l'est d'après le rappel ci-dessus, et que $\mathcal M_{g, [\frac{m}{n}]}\to \mathcal M_{g}$ est géométriquement irréductible, il en est alors de même de $\mathcal M_{g, [m], \underline {\bf kr}}[G]/\Aut(G)$.
		
		\bigskip
		
		Considérons maintenant le cas général. Soit $S \to \mathcal M_{g, [m]}[G]$, correspondant à une courbe $(\mathcal \X/S, \iota)$ et à un diviseur équivariant $D$. On note $D_f$ le diviseur obtenu de $D$ en ne gardant que les points sur lesquels $G$ agit librement, et qui est donc de degré $n r_n(\underline{\bf kr})$. 
		D'après le lemme \ref{r_constante}, la formation de $D_f$ commute au changement de base, et définit 
		donc un morphisme
		\begin{equation}\label{eq:morPsi}
			\Psi \colon \mathcal M_{g, [m], \underline {\bf kr}}[G] \to \mathcal M_{g, [nr_n], \underline {\bf kr}_{red}}[G] 
		\end{equation}
		où $\underline{\bf kr}_{red}$ est défini par une donnée géométrique ${\bf kr}_{red}=({\bf k,r_{red}})\in (\ZZ/n\ZZ)^{\nu}/\mathfrak{S}_{\nu}\times \NN^n$ avec ${\bf r}_{red}={(0, \ldots, 0, r_n)}$.
		
		D'après le cas précédent, le quotient par $\Aut(G)$ de l'espace d'arrivée est irréductible. Pour montrer que le quotient de l'espace de départ l'est également, il suffit donc de montrer qu'il y a au plus un point au-dessus de chaque point générique de $\mathcal M_{g, [n r_n], \underline {\bf kr}_{red}}[G]$. Le lemme suivant, dont nous reportons la preuve en fin de section, réduit cette question à un calcul de degré d'extension : 
		
		\begin{lem}\label{lem:psirep}
			Soit $\underline{\bf kr}$ une donnée de branchement, et $\Psi$ le morphisme défini en \eqref{eq:morPsi}. Alors $\Psi$ est représentable, étale et de degré égal à $\prod_{i < n} \binom{r_i( {\bf k})}{r_i}$, où $({\bf k}, (r_1,\dots,r_n))$ est un relèvement de $\underline{\bf kr}$, le degré étant indépendant du choix du relèvement. 
		\end{lem}

		Considérons un point géométrique générique $\eta$ de $\mathcal M_{g, [n r_n], \underline {\bf kr}_{red}}[G]$ et $\xi \in \Psi^{-1}(\{\eta\})$, et fixons un isomorphisme $G\simeq \mmu_n$, ce qui correspond à fixer un antécédent géométrique $\bf kr$ de $\underline{\bf kr}$. Il suffit d'après le lemme \ref{lem:psirep}, de montrer que l'extension $k(\xi)/k(\eta)$ est de degré au moins $\prod_{i < |G|} \binom{r_i( {\bf k})}{r_i}$. 
		
		Comme $\Psi$ est représentable, il correspond à $\xi$ une courbe $G$-équivariante $\X/k(\xi)$. La description de l'action galoisienne donnée par la proposition \ref{lem:formDiv} étant stable par générisation, il y a pour  le morphisme $\X \to \X/G$ exactement un point de branchement par indice $i$ tel que $r_i \not = 0$ et les corps résiduels $K_i$ correspondants sont linéairement disjoints. On notera $\tilde K_i$ la clôture galoisienne de $K_i/K$.
		
		Considérons l'extension minimale $L/k(\eta)$ rendant rationnels tous les points de branchement. On a alors $L=\tilde K_1 \otimes_{k(\eta)} \ldots \otimes_{k(\eta)} \tilde K_n$  et le groupe de Galois de $L/k(\eta)$ est alors
		\begin{equation*}
			\Gal(L/k(\eta))\simeq\prod_i \mathfrak S_{r_i( {\bf k})}. 
		\end{equation*}
		
		Par ailleurs, considérons l'inclusion $k(\xi) \subset L$. Les extensions $\tilde K_i$ étant linéairement indépendantes, on a ainsi $k(\xi)=\bigotimes_i (\tilde K_i \cap k(\xi))$. La preuve du résultat se réduit alors à montrer que l'extension $\tilde K_i\cap k(\xi)$ de $K$ est de degré au moins égale à $\binom{r_i( {\bf k})}{r_i}$.
		
		Pour un indice fixé $i$, notons $y_1, \ldots, y_{r_i({\bf k})}$ les points de branchements de $\X_L/G$ ayant $i \mod n$ comme données de branchement, numérotés de sorte que le diviseur $D$ corresponde exactement à $\sum_{1 \le \ell \le r_i} y_\ell$.
		
		Tout élément de $\Gal(\tilde K_i/\tilde K_i \cap k(\xi)) \subset \Gal(\tilde K_i/K)=\mathfrak S_{r_i}$ 
		préserve alors les données de Hurwitz, induisant ainsi une bijection de $\{1, \ldots, r_i\}$. On a donc 
		naturellement 
		\begin{equation*}
			\Gal(\tilde K_i/\tilde K_i \cap k(\xi)) \subset Bij(\{1, \ldots, r_i\}) \times Bij(\{r_i+1, \ldots, r_i( {\bf k})\}),
		\end{equation*}
		et il s'ensuit que 
		\begin{equation*}
			[\tilde K_i : \tilde K_i \cap k(\xi)] \le r_i ! (r_i( {\bf k})-r_i)! 
		\end{equation*}
		puis enfin que
		\begin{equation*}
			[\tilde K_i \cap k(\xi):K]=\frac{[\tilde K_i : K]}{[\tilde K_i :\tilde K_i \cap k(\xi)]} \ge \binom{r_i( {\bf k})}{r_i}. 
		\end{equation*}
		Ceci établit la minoration recherchée de l'extension $k(\xi)/k(\eta)$ et termine ainsi la preuve.
	\end{proof}
	
	Donnons maintenant la preuve du lemme restée en attente.
	\begin{proof}[Preuve du lemme \ref{lem:psirep}]
		Le morphisme $\Psi$ s'apparente au morphisme de contraction de Knudsen et est donc représentable d'après \cite{KNUII}.
		Le fait qu'il est étale s'obtient par déformation, et découle du relèvement des sections à l'intérieur du lieu de ramification, lui-même étale sur la base.
		
		Pour ce qui est du degré, il suffit de calculer le cardinal des fibres géométriques. Or
		les points de ces fibres correspondent aux choix, pour chaque indice $i \not = 0 \mod |G|$, de $r_i$
		points ayant une donnée de Hurwitz égale à $i$, et on a le choix parmi $r_i( {\bf k})$.    
	\end{proof}
	
	\begin{rem}\mbox{}
		\begin{enumerate}[label=\roman*),ref=\roman*)]
			\item Les champs des revêtements cycliques \emph{non-ramifiés} apparaissent parmi les composantes irréductibles de $\mathcal M_{g,[m]}[G]/\Aut(G)$. Ceci découle du premier cas dans la preuve du théorème \ref{theo:DescriptionCompIrred} et du résultat original de F.~Catanese déjà cité.

			\item Dans le cas du genre zéro et de $G$ cyclique, les champs des lieux spéciaux $\M_{0,[m]}(G)$ sont irréductibles sur $\QQ$ puisque
			\[
			\M_{0,[m]}(G)=\M_{0,k}/H
			\]
			pour un certain $k$, où $H$ est un sous-groupe de permutations admissibles de points -- voir  \cite{GALGROUPS}. 
		\end{enumerate}
	\end{rem}
	
	\bigskip
	
	Nous disposons ainsi d'une description des composantes irréductibles du champ des lieux spéciaux $\M_{g,[m]}(G)$. Dans la section suivante,  ceci nous permet d'établir l'invariance de chacune des composantes sous l'action du groupe de Galois absolu $\Gq$.
	
	\section{Action du groupe de Galois absolu, torsion géométrique, cas profini de genre $2$}\label{sec:actionGal}
	Grâce à la description des composantes irréductibles du champ $\M_{g,[m]}[G]/\Aut(G)$ obtenue dans la section précédente, on établit dans cette section le deuxième résultat fondamental de cet article, à savoir la forme de l'action galoisienne sur la torsion \emph{géométrique et cyclique} de $\pi_1^{geom}(\M_{g,[m]})$. 
	
	De ce résultat, on déduit la forme de l'action galoisienne sur une partie de la $p$-torsion \emph{profinie cyclique} des courbes de genre $2$.
	
	Dans ce qui suit, nous reprenons les identifications entre groupes fondamentaux et mapping class groups, i.e entre $\pi_1^{orb}(\M_{g,[m]})$ et $\Gamma_{g,[m]}$ d'une part, et entre $\pi_1^{geom}(\M_{g,[m]})$ et $\widehat{\Gamma}_{g,[m]}$ d'autre part, les isomorphismes étant définis à conjugaison près par le choix de points bases dans les groupes fondamentaux.

	\subsection{Torsion géométrique de $\pi_1^{geom}(\M_{g,[m]})$ et action}
	
	Tout d'abord, notons que la stabilité sous l'action galoisienne des classes de conjugaison de torsion géométrique cyclique est une conséquence immédiate des résultats de la section précédente.
	
	\begin{prop}\label{prop:ActConj}
		Soit $G$ un sous-groupe fini cyclique de $\pi_1^{geom}(\M_{g,[m]})$ et $\mathcal G$ une classe de conjugaison fixée de $G$ dans $\Gamma_{g,[m]}$. Alors $\mathcal G$ est stable sous l'action de $\Gq$.
	\end{prop}
	
	\begin{proof}
		D'après la section 2, propositions \ref{prop:galCompIrr} et \ref{prop:irrIsconj}, la stabilité d'une classe de conjugaison de $G$ est équivalente à la stabilité de la composante irréductible correspondante du champ des lieux spéciaux $(\M_{g,[m]}[G]/\Aut(G))_{\bar \QQ}$. 
		
		Suivant le théorème \ref{theo:DescriptionCompIrred}, une telle composante est définie par une donnée de branchement $\underline{\bf kr}$, et de la forme $\left(\M_{g,[m],\underline{\bf kr}}/\Aut(G)\right)_{\bar \QQ}$. Or, d'après la proposition \ref{prop:irrIsStack}, une telle composante $\M_{g,[m],\underline{\bf kr}}/\Aut(G)$ est un champ algébrique défini sur $\QQ$ et est donc fixe sous l'action de $\Gq$. Il s'ensuit la stabilité de chacune des composantes de $(\M_{g,[m]}[G]/\Aut(G))_{\bar \QQ}$, et de chacune des classes de conjugaison de $G$ dans $\Gamma_{g,[m]}$.
	\end{proof}
	
	Autrement dit, pour $\sigma\in\Gq$ et $\gamma\in \pi_1^{geom}(\M_{g,[m]})$ de torsion géométrique, il existe $\rho_{\sigma} \in \pi_1^{geom}(\mathcal M_{g, [m]})$ tel que
	\begin{equation*}
		\sigma(\gamma)=\rho_{\sigma}\ \gamma^{\ell_{\sigma}}\ \rho_{\sigma}^{-1}\qquad \textrm{ pour un certain } \ell_{\sigma} \in \widehat{\ZZ}.
	\end{equation*}
	Ce résultat ne permet néanmoins pas de compléter l'action et d'identifier la puissance $\ell_{\sigma}$. Il s'agit d'un problème de compatibilité d'action galoisienne en des points biens choisis de la composante, dont la résolution occupe le reste de cette section.
	
	\bigskip
	
	\noindent\textsc{Action cyclotomique, compatibilité d'actions galoisiennes.} Nous montrons maintenant que la puissance $\ell_{\sigma}$ ci-dessus s'identifie avec le caractère cyclotomique. 
	
	L'idée consiste à transporter l'action de $\Gq$ en un $K$-point $x_K$ de la composante irréductible déterminée par la classe de conjugaison de $G$, de telle sorte que l'action de $\Gq$ soit compatible avec l'action de $\Gal(\overline K/K)$ sur le groupe d'automorphismes de $x_K\in\M_{g,[m]}[G]$.
	
	Le caractère cyclotomique apparaît alors \emph{via} le branch cycle argument.
	
	\bigskip
	
	L'extension $K$ de $\QQ$ possédant les propriétés galoisiennes voulues est ainsi donnée par le lemme suivant.
	\begin{lem}
		\label{lem:CourbeGen}
		Il existe un morphisme $\Spec K \to \mathcal M_{g, [m], \underline {\bf kr}}[G]/\Aut(G)$
		tel que $K$ et $\bar \QQ$ sont linéairement disjoints sur $\QQ$.
	\end{lem}
	
	\begin{proof} Soit $m'\in\NN$ suffisamment grand pour que $\mathcal M_{g, m'}$ soit représentable. On voit alors
		que $\left(\mathcal M_{g, [m], \underline {\bf kr}}[G]/\Aut(G)\right)\times_{\mathcal M_{g}} \mathcal M_{g, m'}$ l'est également
		et qu'il est de plus géométriquement irréductible sur $\QQ$. Le corps résiduel de son point générique fournit le morphisme annoncé.
	\end{proof}
	
	Un tel $K$-point de $\mathcal M_{g, [m], \underline {\bf kr}}[G]/\Aut(G)$ étant donné, décrivons la situation en ce qui concerne les actions de $\Gq$ et $\Gal(\overline K/K)$ sur les groupes fondamentaux de la courbe $\X$ associée.
	
	En terme de groupes de Galois, si $K/\QQ$ est une extension linéairement disjointe de $\bar \QQ$, on a alors un isomorphisme
	\begin{equation*}
		\Gal(\bar \QQ/\QQ) \overset{\sim}{\to} \Gal(K \otimes_\QQ \bar \QQ/K)
	\end{equation*}
	qui définit ainsi un morphisme surjectif 
	\begin{equation}
		\label{eq:RestrictionAction}
		\Gal(\bar K/K)\to \Gal(K \otimes_\QQ \bar \QQ/K) \simeq \Gal(\bar \QQ/\QQ).
	\end{equation}
	
	\bigskip
	
	Par ailleurs, fixons l'action du groupe de Galois sur le groupe fondamental par le choix d'un point base \emph{tangentiel} à l'infini. Soit $Z$ un $\QQ$-champ muni d'un point base tangentiel $\overrightarrow b$. Notant $\overrightarrow b_K$ le point base tangentiel induit par changement de base à $K$ il existe, suivant le théorème d'invariance du groupe fondamental par changement de base algébriquement clos, un isomorphisme
	\begin{equation*}
		\pi_1^{alg}(Z_{\bar K}, \overrightarrow b_K) \overset{\sim}{\to} \pi_1^{alg}(Z_{\bar \QQ},\overrightarrow b) 
	\end{equation*}
	Ces deux points bases tangentiels définissent alors une action de $\Gal(\bar K/K)$ et de $\Gal(\bar \QQ/\QQ)$ sur $\pi_1^{alg}(Z_{\bar K}, \overrightarrow b_K)$ -- voir section 1, et plus précisément :
	
	\begin{lem}
		\label{lem:ChangementCorps}
		L'action de $\Gal(\bar K/K)$ sur $\pi_1^{alg}(Z_{\bar K}, \overrightarrow b_K)$ est induite par 
		celle de $\Gal(\bar \QQ/\QQ)$ à travers le morphisme \eqref{eq:RestrictionAction}.
	\end{lem}
	
	\begin{proof} Par construction de $\pi_1^{alg}(Z_{\bar \QQ}, \overrightarrow b)$, 
		il suffit de le vérifier au niveau de la catégorie des revêtements étales de $Z_{\bar \QQ}$.
		
		Considérons donc un tel revêtement $W \to Z_{\bar \QQ}$ et un élément $\sigma \in \Gal(\bar \QQ/\QQ)$.
		Soit $\tilde \sigma \in \Gal(\bar K/K)$ une préimage de $\sigma$ par le morphisme \eqref{eq:RestrictionAction}.
		La compatibilité des deux actions provient alors du diagramme suivant dans lequel tous les carrés sont cartésiens
		\begin{equation*}
			\xymatrix{
				& W'_{\bar K}\ar[rr]\ar[ld] \ar'[d][dd] & & Z_{\bar K} \ar[ld] \ar[dd]^{\tilde \sigma}\\
				W' \ar[dd]\ar[rr] & & Z_{\bar \QQ}\ar[dd]^<(0.3){\sigma} \\
				& W_{\bar K} \ar[ld]\ar'[r][rr] & & Z_{\bar K} \ar[ld]\\
				W \ar[rr] & & Z_{\bar \QQ}
			} 
		\end{equation*}
		puisque $W'_{\bar K} \to Z_{\bar K}$ provient de $W' \to Z_{\bar \QQ}$. 
	\end{proof}

	Nous établissons maintenant le résultat principal de cet article, qui précise la proposition \ref{prop:ActConj} sous une certaine condition géométrique. 
	
	Rappelons ainsi qu'en toute généralité, un $G$-revêtement $X\to X/G$ se factorise en un $H$-revêtement $X\to X/H$, où $H$ est le sous-groupe de $G$ engendré par les sous-groupes d'inerties des points de ramification et $X/H\to X/G$ est un revêtement étale. 
	
	Par la suite, fixant $G$ un groupe d'inertie géométrique de $\pi_1^{geom}(\mathcal M_{g, [m]})$ et $\tilde{\mathcal G}=\{\mathcal G\}$ un ensemble de classes de conjugaisons de $G$, on dira que $\tilde{\mathcal G}$ est \emph{sans factorisation étale} si tel est le cas des $G$-revêtements associés.
	
	\bigskip
	
	Énon\c cons enfin.
	
	\begin{theo}\label{theo:CycloGeom}
		Soit $G=\langle\gamma\rangle \subset \pi_1^{geom}(\mathcal M_{g, [m]})$ un groupe d'inertie géométrique cyclique fini, i.e. champêtre, et $\tilde{\mathcal G}=\{\mathcal G\}$ un ensemble de classes de conjugaisons sans factorisation étale. Alors, pour $\gamma\in \tilde{\mathcal G}$, il existe $\rho_{\sigma} \in \pi_1^{geom}(\mathcal M_{g, [m]})$ tel que pour tout $\sigma \in \Gal(\bar \QQ/\QQ)$ on a
		\[\sigma(\gamma)=\rho_{\sigma}\ \gamma^{\chi(\sigma)}\ \rho_{\sigma}^{-1}\] 
		où $\chi$ désigne le caractère cyclotomique.
	\end{theo}

	\begin{proof} 
		Posons $G=\langle\gamma\rangle$ et considérons $\mathcal N_{\gamma}$ le lieu des points associé à la classe de conjugaison de $G$. Suivant \cite{BRO90} théorème 2.1, $\mathcal N_{\gamma}$ est une composante irréductible de $\mathcal M_{g,[m]}(G)$
		qui, suivant le théorème \ref{theo:DescriptionCompIrred}, correspond à l'un des champs $\mathcal M_{g, [m], \underline {\bf kr}}[G]/\Aut(G)$.

		\medskip

		On dispose tout d'abord d'un point rationnel $x_K$ de $\left(\mathcal M_{g, [m], \underline {\bf kr}}[G]/\Aut(G)\right)_K$ tel que l'action de $\Gq$ sur son groupe d'automorphisme inertiel est donnée par l'action de $\Gal(\bar K/K)$. En effet, le lemme \ref{lem:CourbeGen}, assure l'existence d'un point $x_K\colon \Spec K \to  \mathcal M_{g, [m], \underline {\bf kr}}[G]/\Aut(G)$ qui convient pour l'action galoisienne, puisque pour $\sigma\in\Gq$ fixé, comme $K$ et $\bar \QQ$ sont linéairement indépendants, il existe $\tilde \sigma \in \Gal(\bar K/K)$ préimage de $\sigma$  qui fait commuter le diagramme suivant
		\begin{equation*}
			\xymatrix{
				\Spec \bar K \ar[d]_ {x_{\bar K}} \ar[r]^{\tilde \sigma} & \Spec \bar K \ar[d]^{x_{\bar K}} \\
				\left(\mathcal M_{g, [m], \underline {\bf kr}}[G]/\Aut(G)\right)_{\bar K} \ar[r]^{\sigma}&  \left(\mathcal M_{g, [m], \underline {\bf kr}}[G]/\Aut(G)\right)_{\bar K} 
			} 
		\end{equation*}

		Supposons maintenant l'action de $\Gq$ sur $\pi_1^{geom}(\M_{g,[m]})$ fixée par le choix d'un point base tangentiel $\overrightarrow b$. Selon le lemme \ref{lem:ChangementCorps}, cette action se ramène par changement de base de $\QQ$ à $K$ à l'action de $\Gal(\bar K/K)$ sur $\pi_1^{alg}((\mathcal M_{g, [m]})_{\bar K}, \overrightarrow b_K)$. D'après \cite{NOOHI04}, le changement de point base de $\overrightarrow b_K$ à $x_{\bar{K}}$ dans le groupe fondamental de l'inclusion canonique
		\begin{equation*}
			\Aut(x_{\bar{K}})\hookrightarrow  \pi_1^{alg}\left(\left(\mathcal M_{g, [m]}\right)_{\bar K}, x_{\bar{K}}\right)
		\end{equation*}
		modifie l'action par la conjugaison et l'on peut donc supposer le groupe fondamental basé en $x_{\bar{K}}$. Considérant le diagramme $\Gal(\bar K/K)$-équivariant
		\begin{equation*}
			\xymatrix{
				\Aut(x_{\bar{K}}) \ar@{^(->}[r]\ar@{^(->}[rd] & \pi_1^{alg}\left(\left(\mathcal M_{g, [m], \underline {\bf kr}}[G]/\Aut(G)\right)_{\bar K}, x_{\bar{K}}\right)\ar[d] \\
				&  \pi_1^{alg}((\mathcal M_{g, [m]})_{\bar K}, x_{\bar K})
			} 
		\end{equation*}
		il suffit ainsi d'établir que l'action de $\Gal(\bar K/K)$ sur $\Aut(x_{\bar{K}})$ dans le groupe $\pi_1^{alg}((\mathcal M_{g, [m], \underline {\bf kr}}[G]/\Aut(G))_K,  x_{\bar{K}})$ est donnée par cyclotomie.

		\medskip
		
		Considérons donc $\mathcal X$ la courbe associée au point $x_K$, alors par construction $\gamma$ est un automorphisme de $x_{\bar K}$. D'après la proposition \ref{prop:ActConj}, on sait que pour tout $\sigma \in \Gal(\bar K/K)$ on a $\sigma(\gamma)=g\ \gamma^{\ell_{\sigma}}\ g^{-1}$.

		Pour tout point de ramification $P \in \X$, notons $I_P$ le stabilisateur dans $\langle \gamma\rangle$ et $\tau$
		un générateur de $I_P$, donc de la forme $\gamma^{n/a}$ où $n$ est l'ordre de $\gamma$ et $a$ est l'ordre de $\tau$. 
		Le \emph{branch cycle argument} \cite{FRIED}[p. 62] donne alors $\sigma(\tau)=\tau^{\chi(\sigma)}$ et donc $\frac{n}{a}\ell_\sigma=\frac{n}{a}\chi(\sigma) \mod n$. On a donc $\ell_\sigma = \chi(\sigma) \mod a.$
		
		Par suite, comme les stabilisateurs engendrent $\langle \gamma \rangle$ par hypothèses, on trouve que 
		\[\ell_\sigma=\chi(\sigma) \mod n\] 
		ce qui conclut la preuve.
	\end{proof}

	\begin{rem}\mbox{}
		\begin{enumerate}
			\item Lorsque $\gamma$ est l'automorphisme associé à un $G$-revêtement de base de genre $0$, on retrouve ainsi l'énoncé classique du \emph{branch cycle argument}. Le théorème \ref{theo:CycloGeom} est alors vrai sans condition de factorisation puisque la droite projective n'admet pas de revêtement étale.
			\item En particulier, dans le cas du genre $1$, ce résultat étend \emph{à une inertie géométrique elliptique d'ordre quelconque} le résultat établi \emph{pour les éléments d'ordre premier elliptiques} à partir de la théorie de Grothendieck-Teichmüller \cite{Colg1}. 
		\end{enumerate}
	\end{rem}

	\subsection{Torsion profinie de $\widehat{\Gamma}_{2,[m]}$ et action}
	Identifiant $\pi_1^{geom}(\M_{2,[m]})$ avec le complété profini du mapping class group $\widehat{\Gamma}_{2,[m]}$, nous établissons l'action du groupe de Galois absolu sur les éléments de torsion \emph{profinis} de ce dernier. À la différence de l'action galoisienne établie dans la section précédente pour des éléments \emph{géométriques} et \emph{d'ordre quelconque}, ce résultat ne concerne que les éléments \emph{de $p$-torsion}.
	
	\bigskip
	
	Nous suivons l'approche employée dans \cite{Colg0, Colg1}, qui réduit les classes de conjugaisons profinies aux classes de conjugaison finies \emph{via} le choix d'une théorie cohomologique pour les groupes \emph{bons}, afin d'appliquer les résultats d'action galoisienne sur ces classes de conjugaison \emph{géométriques}.

	\begin{theo}[P. Symonds - Théorème 1.1]\label{theo:Sym}
		Soit $G$ un groupe discret, résiduellement fini, i.e. $i:G\hookrightarrow \widehat G$, virtuellement de type $FP$ et tel que $i$ induit un isomorphisme $H^q(\widehat G,\mathbb F_p)\simeq H^q(G,\mathbb F_p)$ en haut degré. Alors les classes de $p$-conjugaisons de $\widehat G$ et de $G$ sont en bijection.
	\end{theo}
	
	La propriété d'isomorphisme du théorème précédent est nommée \emph{propriété de bonté} -- Cf. J.P. Serre \cite{SERREL}. Ce résultat est l'analogue du théorème 1.1 de \cite{SYM07}, dont la preuve s'adapte immédiatement au cas d'un groupe bon selon son auteur. Identifiant $\pi_1^{orb}(\M_{2,[m]})$ à $\Gamma_{2,[m]}$, on obtient ainsi :
	
	\begin{prop}\label{prop:pTorsDisc}
		Soit $m\geqslant 0$. Soit $\gamma$ un élément de $p$-torsion de $\pi_1^{geom}(\M_{2,[m]})$. Alors $\gamma$ est conjugué à un élément de torsion géométrique -- i.e. de $\pi_1^{orb}(\M_{2,[m]})$
	\end{prop}
	
	Suivant le théorème \ref{theo:Sym}, la preuve de ce résultat se réduit à établir la bonté de $\Gamma_{2,[m]}$ -- Cf. \cite{ODA}, ainsi que \cite{LOC11} pour une mise en contexte plus générale. La bonté étant préservée par extension celle-ci se réduit ainsi à celle de $\Gamma_{2,0}$ en considérant les suites exactes d'effacement et de permutation de points -- voir \cite{Colg1} proposition 4.3 pour un raisonnement détaillé en genre $1$ -- à partir de la bonté des groupes de surface. La bonté de $\Gamma_{2,0}$ découle quant à elle du quotient par l'involution de Birman
	\[
	1\to \ZZ/2\ZZ \to \Gamma_{2,0} \to \Gamma_{0,[6]} \to 1 
	\]
	et de la bonté de $\Gamma_{0,[6]}$ (voir \cite{Colg0} proposition 2.3).

	\begin{rem}\mbox{}
		\begin{enumerate}[label=\roman*),ref=\roman*)]
			\item Soit $\gamma\in Tor_{p^n}(\Gamma_{2,[m]})$ un élément de $p$-torsion profini et $\gamma_0\in Tor_{p^n}(\Gamma_{2,[m]})$ un élément géométrique fini de sa classe de conjugaison selon la proposition \ref{prop:pTorsDisc}. Suivant la terminologie du théorème \ref{theo:CycloGeom} pour la torsion \emph{géométrique}, on dira ainsi qu'un élément de $p$-torsion procyclique $\gamma$ est sans factorisation étale si c'est le cas de $\gamma_0$.
			\item Le théorème \ref{theo:Sym}, basé sur la théorie du T-foncteur Lanne, est une généralisation des ordre $p$ aux ordres $p^n$ des résultats analogues de \cite{Colg0,Colg1}, basés sur théorie cohomologique de J.-P. Serre.
		\end{enumerate}
	\end{rem}

	\bigskip
	
	Le troisième résultat fondamental de cet article découle alors de notre premier résultat principal.
	\begin{theo}
		Soit $\gamma$ un élément de torsion d'ordre premier de $\pi_1^{geom}(\M_{2,[m]})$. Alors l'action de $\sigma\in\Gq$ est donnée par $\chi(\sigma)$-conjugaison :
		\begin{equation*}
			\sigma(\gamma)=\rho\ \gamma^{\chi(\sigma)}\ \rho^{-1}\qquad \textrm{ pour un certain } \rho\in \pi_1^{geom}(\M_{2,[m]}).
		\end{equation*}
	\end{theo}
	
	\begin{proof}
		Soit $\gamma$ un tel élément de torsion profinie d'ordre premier, et $\sigma\in \Gq$. Alors suivant la proposition \ref{prop:pTorsDisc}, il est conjugué à un élément $\gamma_0$ fini de $\pi_1^{orb}(\M_{2,[m]})$. En tant qu'automorphisme de courbe de genre $2$, $\gamma_0$ est sans factorisation étale, et il découle du théorème \ref{theo:CycloGeom} que l'action est donnée par $\chi(\sigma)$-conjugaison. Ceci implique le résultat recherché pour l'action sur $\gamma$.
	\end{proof}

\bigskip

\bibliographystyle{cdraifplain}
\def\bysame{\leavevmode ---------\thinspace}
\providecommand{\og}{<<~}\providecommand{\fg}{~>>}
\def\cdrandname{\&}
\providecommand\cdrnumero{no.~}
\providecommand{\cdredsname}{eds.}
\providecommand{\cdredname}{ed.}
\providecommand{\cdrchapname}{chap.}
\providecommand{\cdrmastersthesisname}{Memoir}
\providecommand{\cdrphdthesisname}{PhD Thesis}

\end{document}